\documentclass[11pt,a4paper]{article}
\usepackage{graphicx} 

\usepackage[utf8]{inputenc} 
\usepackage[USenglish]{babel} 
\usepackage{graphicx} 
\usepackage{comment} 
\usepackage{amsmath}
\usepackage{amsfonts}
\usepackage{amssymb}\usepackage{amsmath} 
\usepackage{amssymb}
\usepackage[onehalfspacing]{setspace}
\usepackage{mathtools} 
\usepackage{stmaryrd}

\usepackage{amsthm}
\usepackage{abstract} 
\usepackage{float}
\usepackage[mathscr]{euscript}
\usepackage[singlelinecheck=false,justification=RaggedRight,format=plain]{caption}
\usepackage{subcaption}
\usepackage{geometry}
\usepackage{mathabx}
\usepackage{tikz-cd}
\usepackage{xcolor}
\usepackage{array}

\usepackage[hidelinks]{hyperref}

\newcommand{\N}{\mathbb{N}}
\newcommand{\Z}{\mathbb{Z}}
\newcommand{\R}{\mathbb{R}}
\newcommand{\C}{\mathbb{C}}
\newcommand{\X}{\mathfrak{X}}
\DeclareMathOperator{\id}{id}

\newcommand{\del}{\partial}
\newcommand{\dd}{\mathrm{d}}
\newcommand{\ii}{\mathrm{i}}

\DeclareMathOperator{\im}{im}

\DeclareMathOperator{\End}{End}
\DeclareMathOperator{\aut}{Aut}

\DeclareMathOperator{\Ad}{Ad}

\DeclareMathOperator{\Exp}{Exp}
\newcommand{\dbar}{\bar{\partial}}
\newcommand{\A}{\mathcal{A}}

\DeclareMathOperator{\Diff}{Diff}
\DeclareMathOperator{\Hom}{Hom}
\DeclareMathOperator{\coker}{coker}
\DeclareMathOperator{\pr}{pr}

\DeclareMathOperator{\Hproj}{H}
\DeclareMathOperator{\G}{G}
\DeclareMathOperator{\Q}{Q}

\makeatletter
\g@addto@macro\bfseries{\boldmath}
\makeatother

\newtheorem{theorem}{Theorem}[section]
\newtheorem{corollary}[theorem]{Corollary}
\newtheorem{lemma}[theorem]{Lemma}
\newtheorem{proposition}[theorem]{Proposition}

\theoremstyle{definition}
\newtheorem{definition}[theorem]{Definition}
\newtheorem{example}[theorem]{Example}
\newtheorem{rem}[theorem]{Remark}

\addto\captionsUSenglish{}

\title{Deformation theory of complex structures via Hamilton--Nash--Moser theory}
\author{Vicente Cortés and Paula Naomi Pilatus}
\date{}

\begin{document}
\newgeometry{inner=2.50cm, outer=2.50cm, top=2.50cm, bottom=3.50cm}
\maketitle

\begin{abstract}
     We revisit the classical proof by Kuranishi on the existence of a locally complete family of complex deformations on a compact complex manifold using Hamilton's Nash--Moser implicit function theorem. This is relevant to existing and future extensions of this theorem to generalized complex structures on Courant algebroids. 

    \medskip\noindent
{\it MSc classification:} 32G05 (Deformations of complex structures); 58D27 (Moduli problems for differential geometric structures)

\end{abstract}

\section{Introduction}

The foundations for the deformation theory of complex structures were established by Kodaira and Spencer (with contributions by Nirenberg) in~\cite{kodaira1958existence,kodaira1958completeness,KSdefIandII,KSdefIII} who were influenced by Fröhlicher and Nijenhuis~\cite{Frohlicher1957}. Building on these works, in~\cite{Kuranishi1962}, Kuranishi proved a general  theorem on the existence of locally complete families of deformations on compact complex manifolds. Subsequently, in~\cite{kuranishi64}, he gave a simplified proof of his theorem.

In~\cite{Gualtieri:2003dx,Gualtieri2011}, Gualtieri followed Kuranishi's proof strategy in~\cite{kuranishi64} to obtain a far reaching extension of Kuranishi's result to \emph{generalized} complex structures on the generalized tangent bundle $\mathbb{T}M=TM\oplus T^*M$ of a compact manifold $M$. \emph{Generalized complex structures} on $\mathbb{T}M$ are defined as endomorphisms on $\mathbb{T}M$ that square to minus the identity and are compatible with some extra data on $\mathbb{T}M$, known as an Courant algebroid structure. They encompass both symplectic and complex structures as special cases.

 In future work, we aim to extend Gualtieri's deformation theorem to generalized complex structures on a larger class of Courant algebroids whose underlying vector bundle can be different from the generalized tangent bundle. In order to achieve this, some analytic details in the classical work by Kuranishi~\cite{kuranishi64} (which are also relevant to \cite{Gualtieri:2003dx,Gualtieri2011}) need to be revisited.

Kuranishi's proof in~\cite{kuranishi64} relies on Sobolev spaces and refers to both the inverse and implicit function theorem on Banach spaces. However, the presentation in~\cite{kuranishi64} omits proofs of the diﬀerentiability of the relevant maps, which is required to argue that the implicit function theorem is indeed applicable. One of the difficulties arising is that the action of diffeomorphisms is not smooth in the Banach sense.

In this work, we will set the foundations for our extension of the deformation theory of generalized complex structures to larger classes of Courant algebroids by revisiting Kuranishi's proof in~\cite{kuranishi64} in the framework of tame Fréchet spaces using the Nash--Moser inverse function theorem as well as Hamilton's Nash--Moser implicit function theorem. This
approach has the advantage of allowing us to remain entirely within the smooth category throughout
the proof, avoiding the need to track diﬀerent Sobolev spaces or invoke additional regularity arguments. Along the way we will furthermore provide the missing analytic details mentioned above.

We would like to point out that Hamilton--Nash--Moser theory has already been applied to questions of complex deformation theory: In~\cite{hamilton1977deformation}, Hamilton used an extension of the Nash--Moser theorem to prove that given a relatively compact open subset $X\subset Y$ (with smooth boundary) of a complex manifold $Y$, under certain assumptions, every deformation of the complex structure on $X$ arises from a small perturbation of $X$ in $Y$. In~\cite{Donaldson2025}, Donaldson and Lehmann used Hamilton--Nash--Moser theory to study certain deformations of Calabi-Yau three-folds with boundary.

This paper is structured as follows: In Section~\ref{sect: preliminaries cx geom}, we set up our notation and recall some results and formulas from complex geometry as well as Hodge theory. In Section~\ref{sect: prelim HNM}, we summarize the the concepts and results from Hamilton--Nash--Moser theory that are relevant to this paper. In Section~\ref{sect: diffeo action on defs}, we study the action of small diffeomorphisms on infinitesimal deformations and show that it is described by a smooth tame map. Moreover, we will give an explicit formula for the derivative of this map, which will be essential in our proof of the local completeness of the Kuranishi family. Finally, in Section~\ref{sect: def thm}, we prove Kuranishi's theorem using the Nash--Moser inverse function theorem as well as Hamilton's Nash--Moser implicit function theorem.

\subsection*{Acknowledgements}
The authors have received funding from the Deutsche Forschungsgemeinschaft 
(DFG, German Research Foundation) under Germany's Excellence Strategy, EXC 2121 ``Quantum Universe,'' 390833306 and under -- SFB-Gesch\"aftszeichen 1624 -- Projektnummer 506632645. We thank Jan Heck, Dominik Gutwein and Markus Röser for valuable comments on the paper. Our presentation has benefitted from the fresh point of view of Gualtieri in~\cite{Gualtieri:2003dx,Gualtieri2011} on Kuranishi's proof~\cite{kuranishi64}.

\section{Preliminaries on complex geometry}
\label{sect: preliminaries cx geom}

Given an almost complex manifold $(M,J)$, we denote the $\pm \ii$-eigen-bundles of $J$ by $T^{1,0}(M,J)$ and $T^{0,1}(M,J)$, respectively, and the $(p,q)$-bundles by $\Lambda^{p,q}(M,J)$. Moreover, for some open subset $U\subset M$, we denote the smooth $(p,q)$-forms on $U$ with respect to $J$ by $\mathcal{A}^{p,q}_{M,J}(U)$. Similarly, if $(M,J)$ is a complex manifold, given a holomorphic vector bundle $E\to M$, we denote the space of $(p,q)$-forms with values in $E$ by $\A^{p,q}_{M,J}(U,E)$. 

In the following and later on in the text, we will often fix a manifold $M$ and consider the gradings with respect to a fixed complex structure $J$. We will then often omit writing $M$ and $J$ and instead denote e.g.\ $T^{1,0}(M,J)=T^{1,0}$, $\Lambda^{p,q}(M,J)=\Lambda^{p,q}$, $\A^{p,q}_{M,J}=\A^{p,q}$ etc.

\subsection{Forms with values in the holomorphic tangent bundle}
\label{sect: forms with values in holom}
In the following, we fix a complex manifold $(M,J)$ and consider $(p,q)$-forms with values in the holomorphic tangent bundle.

We will denote the Dolbeault operator of $T^{1,0}$ by $\dbar$. Furthermore, we define a bracket
 \begin{align*}
     [\cdot,\cdot]\colon \A^{0,p}(T^{1,0})\times \A^{0,q}(T^{1,0})\longrightarrow \A^{0,p+q}(T^{1,0})
 \end{align*}
by applying the wedge product on the form parts and the complexified Lie bracket on the $T^{1,0}$-parts. More precisely, in local coordinates, we may write $\alpha=\dd \bar{z}^I\otimes \alpha_I^{\;j}\partial_j$, $\beta=\dd \bar{z}^K\otimes \beta_K^{\;l}\partial_l$, where $|I|=p$, $|K|=q$, and obtain
\begin{align*}
    [ \alpha,\beta ]:=\dd \bar{z}^I\wedge \dd \bar{z}^K\otimes [\alpha_I^{\;j}\partial_j, \beta_K^{\;l}\partial_l],
\end{align*}
where we used multi-index notation and the Einstein convention.
The bracket can be defined more generally on $(p,q)$-forms. The following lemma lists some properties of the bracket and Dolbeault differential:

\begin{lemma}
\label{lem: properties dbar and bracket}
    Let $U\subset M$, $\varphi\in \mathcal{A}^{0,p}(U,T^{1,0})$, $\psi\in \mathcal{A}^{0,q}(U,T^{1,0})$, $\eta\in \A^{0,r}(U,T^{1,0})$. Then
    \begin{enumerate}
        \item $[\varphi,\psi]=(-1)^{pq+1}[\psi,\varphi]$;
\item $\dbar [\varphi,\psi]=[\dbar \varphi,\psi]+(-1)^p[\varphi,\dbar\psi]$;
        \item $(-1)^{pr}\left[\varphi,[\psi,\eta]\right]+(-1)^{pq}\left[\psi,[\eta,\varphi]\right]+(-1)^{qr}\left[\eta,[\varphi,\psi]\right]=0$.
\end{enumerate}
\end{lemma}

We will furthermore need the following explicit formulas for the brackets and Dolbeault derivatives of $0$- and $(0,1)$-forms with values in $T^{1,0}$, which can be directly computed using the local expressions.
\begin{lemma}
\label{lem: formulas dbar and bracket for 0 and 0,1}
    Let $\omega\in\A^{0,1}\left(M,T^{1,0}\right)$, $\xi\in \A^{0}\left(M,T^{1,0}\right)$, $X,Y\in \Gamma\left(M,T^{0,1}\right)$. Then
    \begin{enumerate}
        \item $\dbar\xi(X)=\pr_{1,0}[X,\xi]$;
        \item $\dbar\omega(X,Y)=\pr_{1,0}[X,\omega(Y)]-\pr_{1,0}[Y,\omega(X)]-\omega([X,Y])$;
        \item $[\omega,\xi](X)=[\omega(X),\xi]-\omega\left(\pr_{0,1}[X,\xi]\right)$;
        \item $[\omega,\omega](X,Y)=2[\omega(X),\omega(Y)]-2\omega\left(\pr_{0,1}[\omega(X),Y]\right)-2\omega\left(\pr_{0,1}[X,\omega(Y)]\right)$.
    \end{enumerate}
\end{lemma}

\subsection{Deformations of complex structures}
Let now $(M,J)$ be a compact complex manifold. 

\begin{definition}
An almost complex structure $J'$ on $M$ has \emph{finite distance} to $J$ if there exists some $\omega_{J'}\in \mathcal{A}^{0,1}(M,T^{1,0})$ such that for every $p\in M$
\begin{align*}
    T^{0,1}(J')_p=\left\{-\omega_{J'}(X_p)+X_p : X_p\in T_p^{0,1}\right\}.  
\end{align*}
\end{definition}

Conversely, given $\omega\in \A^{0,1}(M,T^{1,0})$, we can define an almost complex structure $J_{\omega}$ on $M$ with finite distance to $J$ if the subbundle 
$$
    L_{\omega}=\left\{-\omega(X)+X: X\in T^{0,1}\right\}\subset T_{\C}  
$$
satisfies the condition
\begin{align}
\label{eq: condition almost complex}
(L_{\omega})_p\cap\overline{(L_{\omega})_p}=\{0\}\quad \forall p\in M   . 
\end{align}
Then $\omega$ defines an almost complex structure on $M$ via $T^{0,1}(J_{\omega}):=L_{\omega}$. For small enough $\omega\in \A^{0,1}(M,T^{1,0})$ the condition \eqref{eq: condition almost complex} is always satisfied. 
\begin{definition}
    We will call a section $\omega\in\A^{0,1}(M,T^{1,0})$ satisfying \eqref{eq: condition almost complex} an \emph{almost complex deformation} of the complex structure $J$.
\end{definition}

It was shown in \cite{kodaira1958existence} that the condition of the deformed structure to be integrable translates into a partial differential equation:

\begin{theorem}[Equation (3') in \cite{kodaira1958existence}]
    The almost complex structure $J_{\omega}$ obtained from a complex structure $J$ and an almost complex deformation $\omega\in \A^{0,1}(M,T^{1,0})$ is integrable if and only if
    \begin{align}
    \label{eq: Maurer-Cartan}
        \dbar \omega -\tfrac{1}{2}[\omega,\omega]=0.
    \end{align}
\end{theorem}

\begin{definition}
    An almost complex deformation $\omega$ is called a \emph{complex deformation} if it satisfies equation~\eqref{eq: Maurer-Cartan}.
\end{definition}

\begin{definition}
    A \emph{(smooth) family of almost complex deformations} of a complex structure $J$ is a smooth map 
    \begin{align*}
        \omega\colon S\longrightarrow \A^{0,1}(M,T^{1,0})
    \end{align*}
    for some open neighborhood $S$ of $0$ in some finite dimensional vector space such that $\omega(0)=0$ and such that for every $s\in S$, the form $\omega(s)$ satisfies condition \eqref{eq: condition almost complex}. Smoothness is defined here as smoothness of $\omega$ considered as a map $M\times S \to \Lambda^{0,1}T^*\otimes T^{1,0}$. The family $\{\omega(s):s\in S\}$ is a \emph{(smooth) family of complex deformations} if in addition for every $s\in S$, the form $\omega(s)$ satisfies equation~\eqref{eq: Maurer-Cartan}.
\end{definition}

The group of diffeomorphisms acts on the space of almost complex structures. We will see in Section~\ref{sect: diffeo action on defs} that for a diffeomorphism $f$ with small enough one-jet, its action on an almost complex structure $J_{\omega}$ with finite distance from $J$ has again finite distance from $J$ and we will denote the corresponding almost complex deformation by $f\cdot \omega$. We will consider two (almost) complex deformations related in this way by a diffeomorphism as equivalent.

\begin{definition}
    We say that a family of complex deformations $\omega'\colon S'\longrightarrow \A^{0,1}(M,T^{1,0})$ is \emph{obtainable from} the family of complex deformations $\omega\colon S\longrightarrow \A^{0,1}(M,T^{1,0})$
   if there is a smooth map $\tau\colon S'\rightarrow S$ such that $\tau(0)=0$ and a family of diffeomorphisms $\{f_{s'}\}_{s'\in S'}$ of $M$ depending smoothly on $s'$, with $f_0=\id_M$ such that for every $s'\in S'$
    \begin{align*}
        \omega(\tau(s'))=f_{s'}\cdot\omega'(s').
    \end{align*}
\end{definition}
\begin{definition}
    A family of complex deformations $\{\omega(s):s\in S\}$ is called \emph{locally complete} (or \emph{versal} \cite{arnold}) if for every family of complex deformations the restriction to some neighborhood of $0$ is obtainable from $\{\omega(s):s\in S\}$.
\end{definition}

\begin{rem}
In \cite{kuranishi64}, Kuranishi instead defined the notion of \emph{complex analytic families of deformations}, where $S$ is allowed to be an analytic set in some open neighborhood in a finite dimensional complex vector space and the family $\omega(s)$ should depend holomorphically on $s\in S$ in an appropriate sense. 
\end{rem}

\subsection{Hodge theory for compact Hermitian manifolds}
\label{sect: hodge thy}

Let now $E$ be a holomorphic vector bundle over a compact complex manifold $(M,J)$. We fix a hermitian metric on $M$ and a hermitian bundle metric on $E$. 
These induce hermitian bundle metrics $\langle\cdot,\cdot\rangle$ on $\Lambda^{p,q}\otimes E$. We define a hermitian product on the spaces $\A^{p,q}(M,E)$, via
 \begin{align*}
    (\alpha,\beta):=\int_M \langle\alpha,\beta\rangle\,\mathrm{vol},
\end{align*}  
where $\mathrm{vol}$ denotes the metric volume form.

Let $\dbar^*$ be the formal adjoint of the Dolbeault operator $\dbar$ with respect to $(\cdot,\cdot)$. We denote by
\begin{align*}
        \Delta:=\Delta_{\dbar}=\dbar\dbar^*+\dbar^*\dbar
    \end{align*}
the Laplace operator and by
    \begin{align*}
    \mathcal{H}^{k}(M,E) &:= \{\alpha\in \A^{k}(M,E):\dbar\alpha=0=\dbar^*\alpha\}=\{\alpha\in\A^k(M,E):\Delta\alpha=0\}\\
    \mathcal{H}^{p,q}(M,E) &:= \{\alpha\in \A^{p,q}(M,E):\dbar\alpha=0=\dbar^*\alpha\}=\{\alpha\in\A^{p,q}(M,E):\Delta\alpha=0\}.
\end{align*}
the spaces of harmonic forms.

\begin{theorem}[Hodge decomposition]
There exists a natural orthogonal decomposition
\begin{align*}
    \A^{p,q}(M,E)=\dbar \A^{p,q-1}(M,E)\oplus \mathcal{H}^{p,q}(M,E)\oplus \dbar^*\A^{p,q+1}(M,E).
\end{align*}
Moreover, the canonical projection $\mathcal{H}^{p,q}(M,E)\rightarrow H^{p,q}(M,E)$ onto the Dolbeault cohomology is an isomorphism.
\end{theorem}
We denote by $\Hproj\colon \A^{\bullet}\rightarrow \mathcal{H}^{\bullet}$ the orthogonal projection and by $\G\colon\A^{\bullet}\rightarrow\A^{\bullet}$ the Green's operator satisfying 
    \begin{align*}
        \im \G=\ker\Hproj,\quad\ker \G=\im \Hproj,\quad \Hproj+\G\Delta=\id.
    \end{align*}
Furthermore, we define the operator $\Q:=\dbar^*\G$.
We will frequently use the following standard identities (see e.g.~\cite{Wells2007-er}).
\begin{proposition}
\label{prop: H G and Q}
    The operators $\Hproj$, $\G$ and $\Q$ satisfy 
    \begin{enumerate}
        \item $\G\dbar=\dbar \G$, $\G\dbar^*=\dbar^*\G$;
        \item $\Hproj+\dbar \Q+\Q\dbar=\id$;
        \item $\Q^2=\dbar^*\Q=\Q\dbar^*=\Hproj \Q=\Q\Hproj=0$.
    \end{enumerate}
\end{proposition}

\section{Preliminaries on Hamilton--Nash--Moser theory}
\label{sect: prelim HNM}
In this section we review the needed concepts and facts from Hamilton--Nash--Moser theory \cite{hamilton82}. 

\subsection{Tame Fréchet spaces and the Nash--Moser category}

\subsubsection{Graded Fréchet spaces}
\label{sect: graded frechet}
\begin{definition}
    A \emph{Fréchet space} is a topological vector space $\mathcal{F}$ the topology of which 
    is induced by a countable family of semi-norms $\{\Vert\cdot\Vert_n\}_{n\in \mathbb{N}_0}$ (in the sense that 
    the sets $\{ x\in \mathcal F \mid \| x\|_n < \varepsilon \}$, $n\in \N_0$, $\varepsilon>0$,  form a neighborhood subbasis of the origin) such that the following properties are satisfied:\vspace{-7mm}
    \begin{itemize}\singlespacing
    \singlespacing
        \item \emph{Hausdorff property}: $\Vert f\Vert_n=0$ $\forall\,n\in\N_0$ $\implies$ $f=0$; 
        \item \emph{completeness}: every sequence $\{f_i\}\subset \mathcal{F}$ such that $\Vert f_i-f_j\Vert_n\to 0$ for every $n\in\N_0$ as $i,j\to\infty$ converges.   
    \end{itemize}
\end{definition}

\begin{definition}
    A \emph{grading} on a Fréchet space $\mathcal{F}$ is a collection $\{\Vert\cdot\Vert_n\}_{n\in\N_0}$ of semi-norms on $\mathcal{F}$ inducing the topology such that for every $f\in \mathcal{F}$
    $$
        \Vert f\Vert_0\leq \Vert f\Vert_1\leq \Vert f\Vert_2\leq...
    $$
    Two gradings $\{\Vert\cdot\Vert_n\}$ and $\{\Vert\cdot \Vert'_n\}$ on a Fréchet space $\mathcal{F}$ are \emph{tamely equivalent} of degree $r\in \Z$ and base $b\in\N_0$ if for all $f\in \mathcal{F}$ and $n\geq b$ we have
    $$
        \Vert f\Vert_n\leq C(n)\Vert f\Vert'_{n+r}\quad \text{and}\quad \Vert f\Vert'_n\leq C'(n)\Vert f\Vert_{n+r}
    $$
    for some constants $C(n), C'(n)> 0$.
    A \emph{graded Fréchet space} is a Fréchet space with a choice of grading.
\end{definition}

The prototypical example~\cite{hamilton82} of a graded Fréchet space is the space $\left(\Sigma(\mathcal{B}),\{\Vert\cdot\Vert_n\}_{n\in \N_0}\right)$
of rapidly decreasing sequences $\{f_k\}$  in a Banach space $(\mathcal{B}, \Vert\cdot\Vert_\mathcal{B})$, that is sequences such that 
$$
    \Vert\{f_k\}\Vert_n:=\sum_{k=0}^{\infty}\, e^{nk}\Vert f_k\Vert_\mathcal{B}<\infty 
$$
for every $n$.

The most important example of a (graded) Fr\'echet space for us is the space $\Gamma(M,V)$ of smooth global sections of a vector bundle $V$ over a compact oriented manifold $M$.
Fixing  bundle metrics and connections on $V$ and $TM$ and denoting by $D^j s$ the $j$-th covariant derivative of $s\in \Gamma(V)$ the $W^{k,2}$-Sobolev norms defined by 
    \begin{align*}
        \Vert f\Vert_{k,2}:=\left(\sum_{i=0}^k\, \int _M\langle D^if,D^if\rangle \,\mathrm{vol}\right)^{1/2}\end{align*}
        constitute a grading on $\Gamma(M,V)$, 
    where $\langle\,,\,\rangle$ denotes the induced bundle metric on $(T^*M)^{\otimes k}\otimes V$ and $\mathrm{vol}$ denotes the metric volume form. 
An equivalent grading is given by the $C^k$-norms
    \begin{align*}
    \Vert s\Vert_{k,\infty} :=\sum_{j=0}^{k}\,\sup_{x\in M} \Vert D^j s(x)\Vert ,   
    \end{align*}
    as follows from the Sobolev embedding theorem.

\subsubsection{Directional derivatives and families of linear maps}\label{sect: derivatives and families of linear maps}
\begin{definition}
    Let $\mathcal{F}$ and $\mathcal{G}$ Fréchet spaces, $\mathcal{U}\subset \mathcal{F}$ open and $P\colon \mathcal{U}\rightarrow \mathcal{G}$ a continuous map. The \emph{derivative} of $P$ at $f\in \mathcal{U}$ in the direction $h\in \mathcal{F}$ is defined by
    $$
        DP(f)h=\lim_{t\to 0}\frac{P(f+th)-P(f)}{t}.
    $$
    We say $P$ is \emph{continuously differentiable} or $C^1$ on $\mathcal{U}$ if the limit exists for all $f\in \mathcal{U}$ and all $h\in \mathcal{F}$ and if
    $$
        DP\colon \mathcal{U}\times \mathcal{F}\longrightarrow \mathcal{G}
    $$
    is continuous.
\end{definition}

Here, we use the notation $L(f,g)=L(f)g$ to indicate that the map $L$ is linear in the second variable. 
We also refer to such maps as \emph{families of linear maps}  
and denote families of multi-linear maps 
as $L(f,g_1,...,g_n)=L(f)\{g_1,...,g_n\}$. 

The directional derivative on Fréchet spaces satisfies the usual chain rule \cite{hamilton82}, i.e.~given Fréchet spaces $\mathcal{F}$, $\mathcal{G}$ and $\mathcal{H}$, open subsets $\mathcal{U}\subset \mathcal{F}$, $\mathcal{V}\subset \mathcal{G}$ and $C^1$-maps $P\colon \mathcal{U}\rightarrow \mathcal{V}$, $Q\colon \mathcal{V}\rightarrow \mathcal{H}$, the composition $Q\circ P$ is a $C^1$-map and 
    $$D[Q\circ P](f)h=DQ\left(P(f)\right)DP(f)h.$$

For maps that depend on several variables we can define partial differentiability in the usual way. Maps of class $C^n$ are defined
iteratively and the $n$-th derivative of a $C^n$-map is a family of $n$-linear maps. A map of Fréchet spaces is called \emph{smooth} if it is $C^n$ for every $n\in \N$.

For a $C^1$-map, the following property holds.
\begin{proposition}[Theorem I.3.2.2 in \cite{hamilton82}]
\label{prop: int equation DP}
    Let $P\colon (\mathcal{U}\subset \mathcal{F})\rightarrow \mathcal{G}$ be a $C^1$-map between Fréchet spaces, $f,h\in \mathcal{F}$ such that $f+th\in\mathcal{U}$ for all $t\in [0,1]$. Then
    $$
        P(f+h)-P(f)=\int_0^1\, DP(f+th)h\,\dd t.
    $$
\end{proposition}
Here, $\int_a^b$ denotes the Riemann integral of a continuous map $f\colon [a,b]\rightarrow \mathcal{F}$ to a Fréchet space $\mathcal{F}$, which is defined in \cite[Thm I.2.2.1]{hamilton82}. 
Later on we will use the following simplified version of Taylor's formula:
\begin{theorem}[Theorem I.3.5.6 in \cite{hamilton82}]
\label{thm: taylor}
Let $\mathcal{F}$, $\mathcal{G}$ be Fréchet spaces, $\mathcal{U}\subset \mathcal{F}$ open and $P\colon \mathcal{U}\rightarrow \mathcal{G}$ a $C^2$-map, $f,h\in \mathcal{F}$ such that $f+th\in\mathcal{U}$ for all $t\in [0,1]$.  Then 
   $$ P(f+h)=P(f)+DP(f)h+R(h),$$
where $R$ is a $C^2$-map such that $R(sh)=s^2R(h,s)$ for small enough $s\in\R$, with $R(h,s)$ depending continuously on $s$. If $P$ is smooth, $R(h,s)$ depends smoothly on $s$.
\end{theorem}

\subsubsection{Tame linear maps and tame Fréchet spaces}
\begin{definition}
    Let $b\in \N_0$, $r\in\Z$. A linear map $L\colon \mathcal{F}\rightarrow \mathcal{G}$ of graded Fréchet spaces satisfies a \emph{tame estimate} of degree $r$ and base $b$ if for every $f\in \mathcal{F}$ and $n\geq b$ we have
    \begin{align*}
        \Vert Lf\Vert_n\leq C(n)\Vert f\Vert_{n+r}
    \end{align*}
    for some constant $C(n)>0$. We say $L$ is \emph{tame linear} if it satisfies a tame estimate for some~$r$ and~$b$. We will often leave the dependence of $C(n)$ on $n$ implicit and simply write $C$.
\end{definition}

We include the following examples to be used later. 

\begin{example}[Example I.1.2.2(3) in \cite{hamilton82}]
\label{ex: lin part diff op tame}
    A linear partial differential operator $L$ of order $r$ (on a vector bundle $V$) over a compact manifold is tame linear of degree $r$:  $\Vert Ls\Vert_{n,2}\leq C\Vert s\Vert_{n+r,2}$ and similarly $\Vert Ls\Vert_{n,\infty}\leq C\Vert s\Vert_{n+r,\infty}$, for all $s\in \Gamma(M,V)$. 
\end{example}

\begin{example}
\label{ex: bracket}
    Let $M$ be a compact complex manifold. Then the bracket $[\cdot,\cdot]\colon \A^{0,p}(M,T^{1,0})\times \A^{0,q}(M,T^{1,0})\rightarrow \A^{p+q}(M,T^{1,0})$ as defined in Section~\ref{sect: forms with values in holom} is tame linear in both arguments separately. More precisely, we have $\Vert [\alpha,\beta]\Vert_{n,2}\leq C\Vert \alpha\Vert_{n+1,2}\Vert\beta\Vert_{n+1,2}$ for $n\geq \dim M/2+1$ and similarly $\Vert [\alpha,\beta]\Vert_{n,\infty}\leq C\Vert \alpha\Vert_{n+1,\infty}\Vert\beta\Vert_{n+1,\infty}$ for all $n\in \N_0$.
\end{example}

\begin{example}
\label{ex: Q}
Given a holomorphic vector bundle $E\rightarrow M$ over a compact complex manifold $M$, the operator $\Q$ given in Section~\ref{sect: hodge thy} is tame linear: For every $\xi\in\A^{\bullet}(M,E)$ and $n\in \N_0$ have
$\Vert \Q \xi\Vert_{n,2}\leq C\Vert \xi\Vert_{n-1,2}$ \cite{Wells2007-er}. 
\end{example}

In order to define tame Fréchet spaces, we need the notion of a \emph{tame direct summand}.
\begin{definition}
    Let $\mathcal{F}$ and $\mathcal{G}$ be graded spaces. Then $\mathcal{F}$ is a \emph{tame direct summand} of $\mathcal{G}$ if there are tame linear maps $L\colon \mathcal{F}\rightarrow \mathcal{G}$ and $M\colon \mathcal{G}\rightarrow \mathcal{F}$ such that the composition
    $$\mathcal{F}\xlongrightarrow{L}\mathcal{G}\xlongrightarrow{M}\mathcal{F}
    $$
    is the identity.
\end{definition}

\begin{definition}
    A \emph{tame} Fréchet space is a graded Fréchet space that is a tame direct summand of the Fr\'echet space $\Sigma (\mathcal{B})$ of rapidly decreasing sequences for some Banach space $\mathcal{B}$.
\end{definition}

The Fréchet space $\Gamma(M,V)$ with the (equivalent) gradings discussed in the end of Section~\ref{sect: graded frechet} is a tame Fréchet space \cite[Cor II.1.3.9]{hamilton82}.

\subsubsection{Smooth tame maps and the Nash--Moser category}

\begin{definition}
    Let $\mathcal{F}$ and $\mathcal{G}$ be graded Fréchet spaces,  $P\colon (\mathcal{U}\subset \mathcal{F})\rightarrow \mathcal{G}$ a map and $b\in \N_0$, $r\in\Z$. We say that $P$ satisfies a \emph{tame estimate} of degree $r$ and base $b$ if for all $f\in \mathcal{U}$ and $n\geq b$ it holds
    \begin{align*}
        \Vert P(f)\Vert_n\leq C(n)\left(1+\Vert f\Vert_{n+r}\right)
    \end{align*}
    for some constant $C(n)>0$. We say that $P$ is \emph{tame} if it is defined on an open set, is continuous and satisfies a tame estimate in a neighborhood of each point. 
\end{definition}

This is a generalization of the notion of tame linear maps, since a map between graded Fréchet spaces is a tame linear map if and only if it is linear and tame \cite[Thm II.2.1.5]{hamilton82}. For tame maps of several variables, we allow different degrees in the different variables. More specifically, we say that a map $P(f,g)$ of graded Frechet spaces satisfies a tame estimate of degree $r$ in $f$ and $s$ in $g$ and base $b$ if 
\begin{align*}
    \Vert P(f,g)\Vert_{n}\leq C(n)(1+\Vert f\Vert_{n+r}+\Vert g\Vert_{n+s})
\end{align*}
for all $n\geq b$ for some constants $C(n)$ independent of $f$ and $g$.

For tame families of linear maps, we have the following.
\begin{lemma}[Lemma II.2.1.7 in \cite{hamilton82}]\label{lem: tame estimate family lin maps}
Let $L(f)h$ be a continuous family of linear maps that satisfies a tame estimate of degree $r$ in $f$ and degree $s$ in $h$ in a $\Vert \cdot\Vert_{b+r}$-neighborhood of $f_0$ and for all $h$ in a $\Vert\cdot\Vert_{b+s}$-neighborhood of $0$. Then for all $n\geq b$, $f$ in a $\Vert\cdot\Vert_{b+s}$-neighborhood of $f_0$ and all $h$ it satisfies the estimate
\begin{align*}
    \Vert L(f)h\Vert_n\leq C(n)(\Vert f\Vert_{n+r}\Vert h\Vert_{b+s}+\Vert h\Vert_{n+s})
\end{align*}
for constants $C(n)$ independent of $f$ and $h$.
\end{lemma}

\begin{definition}
    Let $\mathcal{F}$ and $\mathcal{G}$ be graded Fréchet spaces. A map $P\colon (\mathcal{U}\subset \mathcal{F})\rightarrow \mathcal{G}$ is \emph{smooth tame} if it is smooth and all derivatives of $P$ are tame.
\end{definition}

The set of tame Fréchet spaces together with the smooth tame maps form a category, the \emph{Nash-Moser category}. An important class of smooth tame maps are \emph{non-linear partial differential operators}:
\begin{definition}
    Let $M$ be a compact manifold, $V$ and $W$ vector bundles over $M$. A \emph{non-linear partial differential operators} of order $r$ is a map
    $$P\colon (\mathcal{U}\subset \Gamma(M,V))\longrightarrow \Gamma(M,W)$$
    which factorizes through a smooth fiber-preserving map $U\subset J^r V \to W$, where $U$ is an open subset of the vector bundle $J^r V$ of $r$-jets of sections of $V$. We will also refer to non-linear partial differential operators of order zero as \emph{non-linear vector bundle operators}.
\end{definition}

\begin{proposition}[Theorem II.2.2.7 in \cite{hamilton82}]
\label{prop: nonlin part smooth tame}
     A non-linear partial differential operator is smooth tame. 
\end{proposition}

An application of Proposition~\ref{prop: nonlin part smooth tame} is the following proposition.
\begin{proposition}
\label{prop: comp vb homs sm tame}
    Let $M$ be a compact manifold, $V$, $V_1$, $V_2$ and $V_3$ vector bundles over $M$. Then the maps
    \begin{align*}
        \Gamma\left(M,\Hom(V_2,V_3)\right)\times \Gamma\left(M,\Hom(V_1,V_2)\right)&\longrightarrow \Gamma\left(M,\Hom(V_1,V_3)\right)\\
        (A,B)&\longmapsto A\circ B
    \end{align*}
    and 
        \begin{align*}
        \Gamma(M,\aut V)\subset \Gamma\left(M,\Hom(V,V)\right)&\longrightarrow \Gamma\left(M,\aut V)\right)\\
        A&\longmapsto A^{-1}
    \end{align*}
    are smooth tame.
\end{proposition}

\subsection{The Nash--Moser theorem and families of inverses}

We will now state the Nash--Moser theorem, which is a generalization of the inverse function theorem on Banach spaces to the more general Nash--Moser category. 

\begin{theorem}[Nash--Moser theorem, Theorem III.1.1.1 in \cite{hamilton82}]
\label{thm: nash moser}
Let $\mathcal{F}$ and $\mathcal{G}$ be tame Fréchet spaces and $P\colon (\mathcal{U}\subset \mathcal{F})\rightarrow \mathcal{G}$ a smooth tame map. Suppose that the equation for the derivative $DP(f)h=k$ has a unique solution $h=VP(f)k$ for all $f\in \mathcal{U}$ and $k\in \mathcal{G}$ and that the family of inverses $VP\colon \mathcal{U}\times \mathcal{G}\rightarrow \mathcal{F}$ is a smooth tame map. Then $P$ is locally invertible and each local inverse $P^{-1}$ is a smooth tame map.     
\end{theorem}

In order to check that the family of inverses is a smooth tame family of linear maps it is in fact enough to show that it is continuous and tame, as the following proposition tells us:

\begin{proposition}[Theorem II.3.1.1 in \cite{hamilton82}]
\label{prop: inverse family is smooth tame}
Let $L\colon (\mathcal{U}\subset \mathcal{F})\times \mathcal{H}\rightarrow \mathcal{K}$ be a smooth tame family of linear maps. Suppose that $L(f)$ is invertible  for every $f\in \mathcal{U}$ and the family of inverses 
$$
    V\colon (\mathcal{U}\subset \mathcal{F})\times \mathcal{K}\rightarrow \mathcal{H}
$$
is continuous and tame as a map. Then $V$ is a smooth tame family of linear maps.
\end{proposition}

\begin{theorem}[Implicit function theorem, cf.\ Theorem III.3.3.1 in \cite{hamilton82}]
\label{thm: impl
fn thm}
Let $\mathcal{F}$, $\mathcal{G}$ and $\mathcal{H}$ be tame Fréchet spaces and let 
$$
    A\colon (\mathcal{U}\subset \mathcal{F}\times \mathcal{G})\longrightarrow \mathcal{H}
$$
be a smooth tame map. Suppose that 
the partial derivative $D_f A(f,g)$ is surjective, and there is a smooth tame family of linear maps $$
    V\colon (\mathcal{U}\subset \mathcal{F}\times \mathcal{G})\times \mathcal{H}\longrightarrow \mathcal{F}
$$
such that for every $(f,g)\in \mathcal{U}$ and $h\in \mathcal{H}$, we have
$$
    D_f A(f,g)V(f,g)h=h.
$$
Then if $A(f_0,g_0)=0$ for some $(f_0,g_0)\in \mathcal{U}$ we can find neighborhoods $\mathcal{V}_1\subset \mathcal{F}$ and $\mathcal{V}_2\subset \mathcal{G}$ of $f_0$ and $g_0$, respectively, such that for all $g\in \mathcal{V}_2$ we can find some $f\in \mathcal{V}_1$ with $A(f,g)=0$. Moreover the solution $f=B(g)$ is defined by a smooth tame map $B$.
\end{theorem}

\subsubsection{Smooth tame families of elliptic linear partial differential operators}

The following theorem appears in a very similar form in~\cite{hamilton82} as Theorem~II.3.3.3. 
\begin{theorem}[Theorem~II.3.3.3 in \cite{hamilton82}]
\label{thm: elliptic pdo}
    Let $V\rightarrow M$, $W\rightarrow M$ be vector bundles over some compact manifold $M$. Let moreover $\mathcal{F}$ be a tame Fréchet space, $\mathcal{U}\subset \mathcal{F}$ an open subset and let $V_1$, $V_2$ be finite dimensional vector spaces together with continuous linear maps
    $$
        i\colon V_1\rightarrow \Gamma(M,W),\quad  j\colon\Gamma(M,V)\rightarrow  V_2.
    $$
    Consider a smooth tame family of linear maps
    $$
        \ell\colon \mathcal{U}\times \Gamma(M,V)\longrightarrow \Gamma(M,W)
    $$
    such that for every $f\in \mathcal{U}$, $\ell(f)$ is an elliptic linear partial differential operator of degree $r$. Suppose that for some norm on $V_2$ we have $\Vert j(h)\Vert\leq C\Vert h\Vert_{r-1}$ for every $h\in \Gamma(M,V)$. We define the family of linear maps 
    \begin{align*}
         \tilde{\ell}\colon \mathcal{U}\times\Gamma(M,V)\times V_1&\longrightarrow \Gamma(M,W)\times V_2\\
         (f,h,x)&\longmapsto \tilde{\ell}(f)\{h,x\}:= \left(\ell(f)h+i(x),j(h)\right).
    \end{align*}   
    If for some $f_0\in \mathcal{U}$ the linear map $\tilde{\ell}(f_0)$ is invertible, there is an open neighborhood $\tilde{\mathcal{U}}\subset \mathcal{U}$ of $f_0$ such that $\tilde{\ell}(f)$ is an isomorphism for every $f\in\tilde{\mathcal{U}}$ and the corresponding family of inverses is a smooth tame family of linear maps.
\end{theorem}
The following application of Theorem~\ref{thm: elliptic pdo} will be used later.
\begin{proposition}
\label{prop: inverse family from elliptic}
    Let $\mathcal{U}\subset \mathcal{F}$ be an open neighborhood of $0$ in a tame Fréchet space $\mathcal{F}$. Let $M$ be a compact complex manifold and
    $$A\colon \mathcal{U}\times \A^{0,p}(M,T^{1,0})\longrightarrow \A^{0,p}(M,T^{1,0})$$
    a smooth tame family of linear maps such that
    $$A(\omega)\alpha=\alpha+\varphi(\omega)\alpha,$$
    where $\varphi(\omega)$ is 
    such that $\alpha\mapsto\Delta \varphi(\omega)\alpha$ is a linear partial differential operator of degree $\leq2$, $\varphi(\omega)\alpha\in \im \dbar^*$ for every $(\omega,\alpha)\in \mathcal{U}\times \A^{0,p}(M,T^{1,0})$ and such that $\varphi(0)=0$. Then there is an open neighborhood $\mathcal{U}'\subset \mathcal{U}$ of $0$ such that the restriction
    $$A(\omega)\colon \left(\im \dbar^*\cap \A^{0,p}(M,T^{1,0})\right)\longrightarrow \left(\im \dbar^*\cap \A^{0,p}(M,T^{1,0})\right)$$
    is invertible for every $\omega \in \mathcal{U}'$ and the family of inverses is smooth tame. It follows that $A(\omega)$ is invertible for every $\omega\in \mathcal{U}'$ and the family of inverses is smooth tame.
\end{proposition}
\begin{proof}
    Since the restriction $$\Delta\colon \im \dbar^*\longrightarrow\im \dbar^*$$ is invertible, solving $A(\omega)\alpha=\beta$ for $\alpha\in \im \dbar^*$, where $\beta\in \im \dbar^*$ is equivalent to solving $$\Delta A(\omega)\alpha=\Delta \alpha+\Delta \varphi(\omega)\alpha=\Delta \beta=:\tilde{\beta}$$ for $\alpha\in \im \dbar^*$. We define the smooth tame family of linear maps
    \begin{align*}
        \ell\colon \mathcal{U}\times \A^{0,p}(M,T^{1,0})&\longrightarrow \A^{0,p}(M,T^{1,0}),\\
        (\omega,\alpha)&\longmapsto \Delta\alpha+\Delta \varphi(\omega)\alpha.
    \end{align*}
    Since $\varphi(0)=0$, we observe that $\ell(0)=\Delta$ and therefore, there is an open neighborhood $\mathcal{U}'\subset \mathcal{U}$ of $0$ such that $\ell\colon \mathcal{U}'\times \A^{0,p}(M,T^{1,0})$ is a smooth tame family of elliptic operators of degree 2. We define
    \begin{align*}
        \tilde{\ell}\colon \mathcal{U}'\times \A^{0,p}(M,T^{1,0})\times \mathcal{H}^{0,p}(M,T^{1,0})&\longrightarrow \A^{0,p}(M,T^{1,0})\times \mathcal{H}^{0,p}(M,T^{1,0})\\
        (\omega,\alpha,\delta)&\longmapsto \left(\ell(\omega)\alpha+i(\delta),j(\alpha)\right),
    \end{align*}
    where $$i\colon \mathcal{H}^{0,p}(M,T^{1,0})\hookrightarrow \A^{0,p}(M,T^{1,0}),\quad j\colon \A^{0,p}(M,T^{1,0})\rightarrow \mathcal{H}^{0,p}(M,T^{1,0}),$$ denote the inclusion and orthogonal projection map, respectively. Let $\Vert\cdot \Vert$ be the norm on $\mathcal{H}^{0,p}(M,T^{1,0})$ induced by $\Vert\cdot\Vert_{0,2}$ on $\A^{0,p}(M,T^{1,0})$. Then by the Hodge decomposition $$\Vert j(\alpha)\Vert\leq \Vert \alpha\Vert_{0,2}\leq \Vert \alpha \Vert_{1,2}.$$ Since $\ell(0)=\Delta$ and $\ker \Delta=\mathcal{H}^{0,p}(M,T^{1,0})=\coker \Delta$, we see from the definition of $\tilde{\ell}$ that $\tilde{\ell}(0)$ is invertible. Therefore, by Theorem~\ref{thm: elliptic pdo} it follows that there is a neighborhood $\mathcal{U}''\subset \mathcal{U}'$ such that $\tilde{\ell}(\omega)$ is invertible for every $\omega\in \mathcal{U}''$ and the family $\tilde{p}(\omega)$ of inverses is smooth tame. 

Next we show that as a consequence the restriction $\ell (\omega) : \im \dbar^*\to \im \dbar^*$ is invertible.       
    Since 
    \[ \ell (\omega)|_{\im \dbar^*}= \tilde{\ell}(\omega)|_{\im \dbar^*\oplus0}: \im \dbar^*\oplus0\to \im \dbar^*\oplus0,\] $\ell(\omega)(\im \dbar)\subset \im \dbar\oplus \im \dbar^*$ and $\ell(\omega)(\im \dbar)\cap \im \dbar^*=0$, it follows that $\ell (\omega) : \im \dbar^*\to \im \dbar^*$ is invertible and its family of inverses is a smooth tame family of linear maps given by the restriction 
      $$ p(\omega):= \tilde{p}(\omega)|_{\im \dbar^*\oplus0}\colon \im \dbar^*\cap \A^{0,p}(M,T^{1,0}) \longrightarrow \im \dbar^*\cap \A^{0,p}(M,T^{1,0}).$$
As explained above this implies the invertibility of $A(\omega )|_{\im \dbar^*}: \im \dbar^* \to \im \dbar^*$ with a smooth tame family of inverses. 

    For the last assertion, given $\alpha\in \A^{0,p}(M,T^{1,0})$, $\beta\in \A^{0,p}(M,T^{1,0})$, we write 
    $$\alpha=\alpha_0+\dbar \alpha_1+\dbar^*\alpha_2,\quad \beta=\beta_0+\dbar \beta_1+\dbar^*\beta_2,$$ where $\alpha_0,\beta_0\in \mathcal{H}^{0,p}(M,T^{1,0})$. Then the equation
    $A(\omega)\alpha=\beta$ is equivalent to
    \begin{align*}
        \alpha_0=\beta_0,\quad \dbar\alpha_1=\dbar \beta_1,\quad \dbar^*\alpha _2+\varphi(\omega)\dbar^*\alpha_2=\dbar^*\beta_2-\varphi(\omega)\{\alpha_0+\dbar \alpha_1\}\\
        \quad \implies A(\omega)\dbar^*\alpha_2=\dbar^*\beta_2-\varphi(\omega)\{\beta_0+\dbar \beta_1\},
    \end{align*}
    which has a smooth tame family of solutions by the previous statement.
\end{proof}

\subsection{Tame Fréchet manifolds and tame Lie groups}

\begin{definition}
    A \emph{tame Fréchet manifold} is a Hausdorff topological space with an atlas of coordinate charts taking values in tame Fréchet spaces such that the transition maps are smooth tame maps. A map $P\colon \mathcal{M}\rightarrow\mathcal{N}$ between tame Fréchet manifolds is called a \emph{smooth tame} map of Fréchet manifolds if there exist charts around every point in $\mathcal{M}$ and its image in $\mathcal{N}$ such that the local representative of $P$ in these charts is a smooth tame map of graded Fréchet spaces.
\end{definition}

\begin{definition}
    A \emph{tame Lie group} is a tame Fréchet manifold $\mathcal{G}$ with a group structure such that the multiplication and the inversion map are smooth tame maps. 
\end{definition}

The diffeomorphism group of a compact manifold is a tame Fréchet Lie group \cite{hamilton82}. We will explicitly describe a tame atlas in the following section. Another example of a tame Lie group are the automorphisms of a vector bundle over a compact manifold (cf.\ Proposition~\ref{prop: comp vb homs sm tame}).

\subsubsection{The diffeomorphism group of a compact (complex) manifold}
\label{sect: chart diff(M)}

Let $M$ be a compact manifold. Fix an auxiliary Riemannian metric $h$ on $M$. Then we denote by
$$
    \exp^h\colon TM\longrightarrow M
$$
the exponential map with respect to $h$ and by $\pi\colon TM\rightarrow M$ the canonical projection. There is a tubular neighborhood $U\subset TM$ of $M\subset TM$ on which the extended map
\begin{align*}
    \Exp^h\colon TM&\longrightarrow M\times M\\
    v&\longmapsto \left(\pi(v),\exp^h_{\pi(v)}(v)\right)
\end{align*}
restricts to a diffeomorphism. This gives rise to a homeomorphism from a neighborhood $\mathcal{U}_{\id}\subset \X(M)$ of 0 to a neighborhood $\mathcal{V}_{\id}\subset \Diff(M)$ of $\id\in\Diff(M)$, which we will also denote by 
$$
    \exp^h\colon \mathcal{U}_{\id}\longrightarrow \mathcal{V}_{\id}.
$$
Similarly, for $f\in \Diff(M)$ and some neighborhood $\mathcal{U}_f\subset \Gamma(M,f^*TM)$ of 0 in $\Gamma(M,f^*TM)$, we obtain a homeomorphism
$$
    \exp^h\colon \mathcal{U}_f\longrightarrow \mathcal{V}_f
$$
to a neighborhood $\mathcal{V}_f\subset \Diff(M)$ of $f$. 
\begin{proposition}[Theorem III.2.3.1 and Corollary III.2.3.2 in \cite{hamilton82}]
The local charts $\left(\mathcal{V}_f,{(\exp^h)}^{-1}\right)$ form a smooth tame atlas of $\Diff(M)$. In particular, we may identify $T_f \Diff(M)\cong \Gamma(M,f^*TM)$.  
\end{proposition}

For $f\in \Diff(M)$ we denote by $L_f$ and $R_f$ the left- and right-multiplication by $f$, respectively, and by $I$ the inversion map. 

\begin{proposition}[Example I.4.4.5-6 in \cite{hamilton82}]
\label{prop: DR DL and DI}
For $f,g\in \Diff(M)$, $X\in \Gamma(M,g^*TM)$ we have 
\begin{align*}
    DL_f(g)X=df \circ X,\quad DR_f(g)X=X\circ f,\quad DI(g)X=-dg^{-1}\circ X\circ g^{-1}.
\end{align*}
\end{proposition}
The Lie algebra of the diffeomorphism group is given by $(\X(M),-[\cdot,\cdot])$, where $[\cdot,\cdot]$ is the usual Lie bracket of vector fields. The adjoint action of $\Diff(M)$ on its Lie algebra is given by
\begin{align*}
    \Ad\colon \Diff(M)\times \X(M)&\longrightarrow \X(M)\\
    (f,X)&\longmapsto \Ad(f)X:=df\circ X\circ f^{-1}.
\end{align*}

In our case of interest, the compact manifold is equipped with a complex structure $J$. We will often use the isomorphism of (real) vector bundles 
\begin{align*}
    \Psi\colon T^{1,0}M&\longrightarrow TM \\
    \xi&\longmapsto \xi+\bar{\xi},
\end{align*}
in order to identify $T_f\Diff(M)\cong \Gamma\left(M,f^*T^{1,0}M\right)$.

\begin{lemma}
\label{lem: lemma for approx f on J} 
For $\xi\in\Psi^{-1}(\mathcal{U}_{\id})$ we denote $F_{\xi}:=\exp^h\circ \Psi(\xi)\in\Diff(M)$.
Then 
$$
        \Ad(F_{\xi})X=X-[\Psi\xi,X]+ \tilde{R}(\xi)X,
$$
    where $\tilde{R}(s\xi)=s^2\tilde{R}(\xi,s)$, for a real parameter~$s$, where $\tilde{R}(\xi,s)$ depends smoothly on~$s$ for small~$s$.
\end{lemma}
\begin{proof}
The statement follows directly from Theorem~\ref{thm: taylor} using the facts that $D_0\exp^h=\id$ and 
$$
D_{f}\Ad(f=\id)\{X,Y\}=-[Y,X] 
$$
for $X,Y\in\X(M)$.
\end{proof}

\section{Diffeomorphisms acting on deformations}
\label{sect: diffeo action on defs}

The set of almost complex structures of $M$ can be seen as the space of sections $\mathrm{AC}(M)$ of a fiber subbundle of $\End TM$ and hence is a tame Frechet submanifold of $\Gamma(M,\End TM)$ \cite[Thm II.2.3.1]{hamilton82}.
The diffeomorphism group $\Diff (M)$ of $M$ acts on the  tame Fréchet manifold of almost complex structures on $M$ via
$$
    (f,J)\longmapsto J_{f}:= d f \circ J\circ d f^{-1}.
$$
This is a restriction of the action of diffeomorphisms on sections of the endomorphism bundle of $TM$.
\begin{lemma}
\label{lem: conj with df tame}
    The map 
    \begin{align*}
        \Diff(M)\times \Gamma(M,\End TM)&\longrightarrow \Gamma(M,\End TM)\\
        (f,\psi)&\longmapsto df\circ \psi\circ df^{-1}
    \end{align*}
    is smooth tame. 
\end{lemma}
\begin{proof}
We define the auxiliary map
\begin{align*}
    A\colon\Diff(M)\times \Gamma(M,\End TM)\times \X(M)&\longrightarrow\X(M)\\
    (f,\psi,X)&\longmapsto df\circ \psi\circ df^{-1}(X).
\end{align*}
We observe that
\begin{align*}
    A(f,\psi,X)=&df\circ \psi\circ df^{-1}\circ X\circ f\circ f^{-1}\\
    =&df\circ(\psi\circ (df^{-1}\circ X\circ f))\circ f^{-1}\\
    =&\Ad(f)\{\psi\circ \Ad(f^{-1})X\}.
\end{align*}
It follows that $A$ is a composition of smooth tame maps and hence smooth tame. In a local coordinate neighborhood, we can recover the original map by setting
\begin{align*}
    (f,\psi)\longmapsto \sum_i\, A(f,\psi,\del_i)\otimes \dd x^i.
\end{align*}
These patch together to give a well-defined global section to $\End TM$. The patching can be made explicit using a partition of unity. In particular, this yields a smooth tame map. 
\end{proof}

\begin{proposition}
    The action of the diffeomorphisms on the tame Fréchet manifold $\mathrm{AC}(M)$ of almost complex structures is smooth tame.
\end{proposition}
\begin{proof}
   This follows from Lemma~\ref{lem: conj with df tame}, since $\mathrm{AC}(M)\subset\Gamma(M,\End TM)$ is a tame Fréchet submanifold.
\end{proof}

It was already mentioned in Section~\ref{sect: preliminaries cx geom} that when acting with a diffeomorphism $f$ contained in a small enough $C^1$-neighborhood $\mathcal{V}$ of the identity on a complex structure $J$, the resulting complex structure $J_f$ will again have finite distance from $J$. We will denote by $\Theta(f)\in\A^{0,1}(M,T^{1,0})$ the corresponding complex deformation of $J$. Since $\Theta(f)$ can be thought of as the action of $f\in \Diff(M)$ in a local chart of the tame manifold of almost complex structures, the map $\Theta\colon \mathcal{V}\rightarrow \A^{0,1}(M,T^{1,0})$ is smooth tame.

Instead of acting on the complex structure $J$ we can act with a diffeomorphism $f$ on an almost complex structure $J_{\omega}$ at finite distance from $J$. Again, for $f$ in a small enough $C^1$-neighborhood of the identity, the resulting structure $J_{\omega,f}$ will again have finite distance from $J$ and we will denote the corresponding almost complex deformation by $\Theta(\omega,f)\in\A^{0,1}(M,T^{1,0})$.
Again $\Theta$ can be thought of as a local description of the action of the diffeomorphisms on $\mathrm{AC}(M)$ and hence, it is smooth tame. In the following we will see an explicit formula for $\Theta(\omega,f)$. For this, we introduce the following notation for the different components of the $\C$-linear extension of the differential $df$ of a diffeomorphism $f$ with respect to $T=T^{1,0}\oplus T^{0,1}$ by
    \begin{align*}
        df=\begin{pmatrix}
            \del f &\dbar f\\
            \del \bar{f} & \overline{\del f}
        \end{pmatrix}.
    \end{align*}
\begin{proposition}

    \label{prop: Theta(w,f)}
    \label{prop: Theta(w,f)}
There is a $C^1$-neighborhood $\mathcal{V}$ of the identity in $\Diff(M)$ and a $C^0$-neighborhood $\mathcal{U}\subset \A^{0,1}(M,T^{1,0})$ of $0$ such that for every $f\in \mathcal{V}$, $\omega\in\mathcal{U}$, the almost complex structure $J_{\omega,f}$ has finite distance from $J$. The corresponding almost complex deformation $\Theta(\omega,f)\in\A^{0,1}(M,T^{1,0})$ depends smooth tamely on $\omega$ and $f$ and is given by the formula
\begin{equation}\label{eq: formula theta(omega,f)}
    \Theta(\omega,f)=-\left(\dbar f-\del f\circ \omega\right)\left(\overline{\del f}-\del\bar{f}\circ \omega\right)^{-1},
    \end{equation}
\end{proposition}

\begin{proof}

We only need to prove the explicit formula for $\Theta(\omega,f)$. For this, we observe that
    \begin{align*}
        T^{0,1}(J_{\omega,f})_p=\left\{df\left(Y-\omega(Y)\right):Y\in T^{0,1}_{f^{-1}(p)}\right\}.
    \end{align*}
    Then 
    \begin{align*}
        \pr_{0,1}T^{0,1}(J_{\omega,f})_p=&\left\{\overline{\del f}(Y)-\del\bar{f}\left(\omega(Y)\right):Y\in T^{0,1}_{f^{-1}(p)}\right\},\\
        \pr_{1,0}T^{0,1}(J_{\omega,f})_p=&\left\{\dbar f(Y)-\del f\left(\omega(Y)\right):Y\in T^{0,1}_{f^{-1}(p)}\right\}.
    \end{align*}
    Note that is a $C^0$-neighborhood $\mathcal{U}\subset \A^{0,1}(M,T^{1,0})$ and a $C^1$-neighborhood $\mathcal{V}\subset \Diff(M)$ of the identity such that $\left(\overline{\del f}-\del\bar{f}\circ \omega\right)\colon T^{0,1}\rightarrow T^{0,1}$ is invertible for every $(\omega,f)\in\mathcal{U}\times \mathcal{V}$. Then for every $Y\in T^{0,1}_{f^{-1}(p)}$ we can find a unique $X\in T^{0,1}_p$ such that $$\left(\overline{\del f}-\del\bar{f}\circ \omega\right)Y=X.$$
    It follows that 
    \begin{align*}
        T^{0,1}(J_{\omega,f})_p=\left\{X+\left(\dbar f-\del f\circ \omega\right)\left(\overline{\del f}-\del\bar{f}\circ \omega\right)^{-1}X:X\in T^{0,1}_{(p)}\right\}
    \end{align*}
    and hence
\begin{equation*}
    \Theta(\omega,f)=-\left(\dbar f-\del f\circ \omega\right)\left(\overline{\del f}-\del\bar{f}\circ \omega\right)^{-1}.
    \end{equation*}
\end{proof}

 We will also write $\Theta(\omega,f)=:f\cdot \omega$. 
\begin{lemma}
\label{lem: left action on theta}
    Let $f,g\in \Diff(M)$ and $\omega\in \A^{0,1}(M,T^{1,0})$ such that $\Theta(\omega, g\circ f)$ is defined. Then $\Theta(\omega,g\circ f)=g\cdot\Theta(\omega,f)$. 
\end{lemma}
\begin{proof}
    This follows from 
    \begin{align*}
        J_{g\circ f}=d(g\circ f)\circ J\circ d(g\circ f)^{-1}=dg\circ J_f\circ dg^{-1}=(J_f)_g.
    \end{align*}
\end{proof}
We set
\begin{align*}
\mathcal{U}^{(k)}_{\varepsilon}&:=\{\xi\in\A^0(M,T^{1,0}):\Vert\xi\Vert_{k,\infty}<\varepsilon\},\\
\mathcal{W}^{(k)}_{\varepsilon}&:=\{\omega\in \A^{0,1}(M,T^{1,0}):\Vert \omega\Vert_{k,\infty}<\varepsilon\}.
\end{align*}

Then for $\varepsilon>0$ small enough, we may concatenate $\Theta$ with $\exp^h\circ \Psi$ to obtain a smooth tame map
\begin{align*}
    \mathcal{E}\colon \mathcal{W}^{(0)}_{\varepsilon}\times \mathcal{U}^{(1)}_{\varepsilon}&\longrightarrow \A^{0,1}(M,T^{1,0})\\
    (\omega,\xi)&\longmapsto \mathcal{E}(\omega,\xi)=F_{\xi}\cdot \omega.
\end{align*}

\begin{rem}
    In~\cite{hamilton1977deformation}, Hamilton already stated that a map closely related to $\mathcal{E}$ is smooth tame by arguing that it is a nonlinear partial differential operator. However, this is not clear to us as when substituting $f=\exp^h(\Psi\xi)$ in the formula~\eqref{eq: formula theta(omega,f)} for $\Theta(\omega,f)$, the section $\xi\in \A^{0,1}(M,T^{1,0})$ is evaluated in different base points. This means that the map $\mathcal{E}(\omega,\xi)=\Theta\left(\omega,\exp^h(\Psi\xi)\right)$ does not factorize through the jet bundle. Therefore, we chose to give the different proof above.
\end{rem}
\begin{proposition}
\label{prop: small diff on small alm def}
    Let $\omega\in\mathcal{W}^{(0)}_{\varepsilon}$, and $\xi\in \mathcal{U}^{(1)}_{\varepsilon}$, where $\varepsilon>0$ is small enough so that $\mathcal{E}(\omega,\xi)$ is defined. Then
   \begin{align*}
    \mathcal{E}(\omega,\xi)=\omega-\dbar (a_{\omega}\xi)+[
\omega,a_{\omega}\xi]-
\iota_{\bar{\xi}}\left(\dbar\omega-\tfrac{1}{2}[\omega,\omega]\right)+R(\xi,\omega),  
   \end{align*}
   where $a_{\omega}\xi=\xi+i_{\bar{\xi}}\omega$ and $R(t\xi,\omega)=t^2R(\xi,\omega,t)$ for a real parameter $t$ with $R(\xi,\omega,t)$ depending smoothly on $t$ for small $t$. 
\end{proposition}
\begin{proof}
    We write
    \begin{align*}
        T_p^{0,1}(J_{\mathcal{E}(\omega,\xi)})=&\left\{ \Ad(F_{\xi})\left(X-\omega(X)\right)_p:X\in\Gamma(M,T^{0,1})\right\}\\
        =&\left\{X_p-\omega(X_p)-[\Psi \xi,X-\omega(X)]_p+R_A(\xi)\left(X-\omega(X)\right)_p:X\in\Gamma(M,T^{0,1})\right\},
    \end{align*}
    where we used Lemma~\ref{lem: lemma for approx f on J}. 
    We define $B(\omega,\xi)\colon \Gamma(M,T^{0,1})\rightarrow \Gamma(M,T^{0,1})$ by
    \begin{equation}\label{eq: linear approx B}\begin{aligned}
    B(\omega,\xi)X&=\pr_{0,1}\Ad(F_{\xi})\left(X-\omega(X)\right)\\&=X-\pr_{0,1}[\Psi\xi,X-\omega(X)]+\pr_{0,1}R_A(\xi)\left(X-\omega(X)\right).  \end{aligned}  
    \end{equation}
    For $\varepsilon>0$ small enough, the map $B(\omega,\xi)$ is invertible for every $(\omega,\xi)\in \mathcal{W}_{\varepsilon}^{(0)}\times \mathcal{U}^{(1)}_{\varepsilon}$ and the family of inverses $B(\omega,\xi)^{-1}Z$ depends smooth tamely on $\omega$ and $\xi$. To see this note that the inverse of $\pr_{0,1}\colon T^{0,1}(J_{\mathcal{E}(\omega,\xi)})\rightarrow T^{0,1}$ is given by $\id-\mathcal{E}(\omega,\xi)\colon T^{0,1}\to T^{0,1}(J_{\mathcal{E}(\omega,\xi)})\subset T$, which we already know depends smooth tamely on $\xi$ and $\omega$, and the inverse of $X\mapsto\Ad(F_\xi)(X-\omega(X))$ is given by $Z\mapsto\pr_{0,1}\Ad(F_{\xi}^{-1})Y$. Using that $B(\omega,\xi)^{-1}Z$ depends smoothly on $\xi$ and equation~\eqref{eq: linear approx B}, we note that 
    \begin{align*}
        B(\omega,\xi)^{-1}(Z)=Z+\pr_{0,1}[\Psi\xi,Z-\omega(Z)]+\text{ higher orders of }\xi.
    \end{align*}
    Then for $Z\in\Gamma(M,T^{0,1})$ we find that 
    \begin{align*}
    T_p^{0,1}(J_{\mathcal{E}(\omega,\xi)})=&\left\{Z_p-\omega(B(\omega,\xi)^{-1}(Z)_p)-\pr_{1,0}[\Psi\xi,B(\omega,\xi)^{-1}(Z)-\omega\left(B(\omega,\xi)^{-1}(Z)\right)]_p\right.\\&\quad\quad\quad\left.+\text{ higher orders of }\xi:Z\in\Gamma(M,T^{0,1})\right\}\\
        =&\left\{Z_p-\omega(Z_p)-\omega\left(\pr_{0,1}[\Psi\xi,Z-\omega(Z)]_p\right)-\pr_{1,0}[\Psi\xi,Z-\omega(Z)]_p\right.\\&\quad\quad\left.+\text{ higher orders of }\xi:Z\in\Gamma(M,T^{0,1})\right\}.
    \end{align*}
    It follows that 
    \begin{align}
    \label{eq: for prop small diff on small alm def}
        \mathcal{E}(\omega,\xi)(Z_p)=\omega(Z_p)+\pr_{1,0}[\Psi\xi,Z-\omega(Z)]_p+\omega\left(\pr_{0,1}[\Psi\xi,Z-\omega(Z)]_p\right)+R(\xi,\omega)(Z_p),
    \end{align}
    where $R(t\xi,\omega)=t^2R(\xi,\omega,t)$ for a real parameter $t$ with $R(\xi,\omega,t)$ depending smoothly on $t$ for small $t$. We compute
    \begin{align*}
        \pr_{1,0}&[\Psi\xi,Z-\omega(Z)]_p+\omega\left(\pr_{0,1}[\Psi\xi,Z-\omega(Z)]_p\right)\\=& \pr_{1,0}[\xi,Z]-[\xi,\omega(Z)]-\pr_{1,0}[\bar{\xi},\omega(Z)]+\omega\left([\bar{\xi},Z]\right)+\omega\left(\pr_{0,1}[\xi,Z]\right)-\omega\left(\pr_{0,1}[\bar{\xi},\omega(Z)]\right)\\
        =&-\dbar\xi(Z)+[\omega,\xi](Z)-\dbar\omega(\bar{\xi},Z)+\tfrac{1}{2}[\omega,\omega](\bar{\xi},Z)-\dbar\left(\omega(\bar{\xi})\right)(Z)+\left[\omega,\omega(\bar{\xi})\right](Z)\\
        =&-\dbar(a_{\omega}\xi)(Z)+[\omega,a_{\omega}\xi](Z)-\iota_{\bar{\xi}}\left(\dbar\omega-\tfrac{1}{2}[\omega,\omega]\right)(Z),
    \end{align*}
    where we used Lemma~\ref{lem: formulas dbar and bracket for 0 and 0,1} and wrote $a_{\omega}\xi=\xi+\iota_{\bar{\xi}}\omega$. Plugging this back into equation~\eqref{eq: for prop small diff on small alm def}, we obtain the expression for $\mathcal{E}(\omega,\xi)$ as stated in Proposition~\ref{prop: small diff on small alm def}.
\end{proof}
\begin{lemma}
\label{lem: properties a_omega}
    For some $\varepsilon>0$, $\omega\in\mathcal{W}^{(0)}_{\varepsilon}$, let
    \begin{align*}
        a_{\omega}\colon \A^0(M,T^{1,0})&\longrightarrow \A^0(M,T^{1,0})\\
        \xi &\longmapsto \xi+\iota_{\bar{\xi}}\omega.
    \end{align*}
    Then the family of (real-)linear maps 
    \begin{align*}
        a\colon \mathcal{W}^{(0)}_{\varepsilon}\times \A^{0,1}(M,T^{1,0})&\longrightarrow \A^0(M,T^{1,0})\\
        (\omega,\xi)&\longmapsto a_{\omega}\xi
    \end{align*}
    is smooth tame. Moreover, for $\varepsilon>0$ small enough, $a_{\omega}$ is invertible for every $\omega\in \mathcal{W}^{(0)}_{\varepsilon}$ and the family of inverses $a_{\omega}^{-1}$ is smooth tame.
\end{lemma}
\begin{proof}
That $a$ is smooth tame of tameness degree $0$ follows from Proposition~\ref{prop: nonlin part smooth tame} and the fact that it is a non-linear partial differential operator of degree $0$.
Moreover, $a_{\omega}$ is invertible for $\Vert\omega\Vert_{0,\infty}$ small enough. Define $\tilde{\omega}\colon T^{1,0}\rightarrow T^{1,0}$ by $\tilde{\omega}(\xi):=\omega(\bar{\xi})$. Then $a_{\omega}=\id+\tilde{\omega}$ and the inverse $a_{\omega}^{-1}$ is given by the sum
    \begin{align*}
        a_{\omega}^{-1}=\sum_{k=0}^{\infty}(-\tilde{\omega})^k,
    \end{align*}
    which converges (fiberwise) for small enough $\Vert \omega\Vert_{0,\infty}$. In particular, the whole family of linear maps $a_{\omega}^{-1}\xi$ is a non-linear vector bundle operator and therefore smooth tame by Proposition~\ref{prop: nonlin part smooth tame}.
\end{proof}

We will use Proposition~\ref{prop: small diff on small alm def} to compute $D_{\xi}\mathcal{E}(\omega,\xi)$ for almost complex deformations $\omega$:
\begin{proposition}
\label{prop: DEw}
    Let $\omega$ be an almost complex deformation of $J$. Then
    \begin{align*}
        D_{\xi}\mathcal{E}(\omega,\xi)\eta=-\dbar(b_{(\omega,\xi)}\eta)+[\mathcal{E}(\omega,\xi),b_{(\omega,\xi)}\eta]-\iota_{\overline{c_{\xi}\eta}}\left(\dbar\mathcal{E}(\omega,\xi)-\tfrac{1}{2}[\mathcal{E}(\omega,\xi),\mathcal{E}(\omega,\xi)]\right),
    \end{align*}
    where $b_{(\omega,\xi)}\eta=a_{\mathcal{E}(\omega,\xi)}c_{\xi}\eta$, with $a$ as defined in Proposition~\ref{prop: small diff on small alm def} and 
    \begin{align*}
    c_{\xi}\eta=DC_{\xi}(0)\eta,\quad C_{\xi}\colon \mathcal{U}\ni\tilde{\eta}\mapsto \left(\exp^h\circ \Psi\right)^{-1}\left(F_{\xi+\tilde{\eta}} F_{\xi}^{-1}\right)\in \A^0\left(M,T^{1,0}\right),
    \end{align*}
    for a sufficiently small $C^1$-neighborhood $\mathcal{U}\subset \A^0\left(M,T^{1,0}\right)$ of zero.
\end{proposition}
\begin{proof}

Let $\xi\in \mathcal{U}_{\varepsilon}^{(1)}$. Moreover, let $\mathcal{U}\subset \mathcal{U}_{\varepsilon}^{(1)}$ be an open neighborhood of $0$ such that $\xi+\tilde{\eta}\in\mathcal{U}_{\varepsilon}^{(1)}$ for every $\tilde{\eta}\in\mathcal{U}$. We have:
    \begin{align*}
            \mathcal{E}(\omega,\xi+\tilde{\eta})=F_{\xi+\tilde{\eta}}\cdot\omega.
    \end{align*}
    Note that, possibly after shrinking $\mathcal{U}$, we have
    \begin{align*}
        F_{\xi+\tilde{\eta}}\circ F_{\xi}^{-1}\in \exp^h\circ \Psi(\mathcal{U}_{\varepsilon}^{(1)})
    \end{align*}
    for every $\tilde{\eta}\in\mathcal{U}$. We set
    \begin{align*}
    C_{\xi}(\tilde{\eta}):=\left(\exp^h\circ\Psi\right)^{-1}\left(F_{\xi+\tilde{\eta}} F_{\xi}^{-1}\right).
    \end{align*}
    Then by Lemma~\ref{lem: left action on theta}, we obtain
    \begin{align*}
        \mathcal{E}(\omega,\xi+\tilde{\eta})=\left(F_{C_{\xi}(\tilde{\eta})}\circ F_{\xi}\right)\cdot\omega =F_{C_{\xi}(\tilde{\eta})}\cdot\mathcal{E}(\omega,\xi) =\mathcal{E}\left(\mathcal{E}(\omega,\xi),C_{\xi}(\tilde{\eta})\right).
    \end{align*}
    By Proposition \ref{prop: small diff on small alm def} and since $C_{\xi}(0)=0$, it follows that 
    \begin{align*}
        D_{\xi}\mathcal{E}(\omega,\xi)\eta=&\lim_{t\to 0} 1/t\cdot[\mathcal{E}(\omega,\xi+t\eta)-\mathcal{E}(\omega,\xi)]\\
        =&\lim_{t\to 0} 1/t\cdot[\mathcal{E}\left(\mathcal{E}(\omega,\xi),C_{\xi}(t\eta)\right)-\mathcal{E}(\omega,\xi)]
        \\=&\lim_{t\to0} 1/t\cdot\left[-\dbar(a_{\mathcal{E}(\omega,\xi)}C_{\xi}(t\eta))+[\mathcal{E}(\omega,\xi),a_{\mathcal{E}(\omega,\xi)}C_{\xi}(t\eta)]\right.\\&\qquad\left.-\iota_{\overline{C_{\xi}(t\eta)}}\left(\dbar\mathcal{E}(\omega,\xi)-\tfrac{1}{2}[\mathcal{E}(\omega,\xi),\mathcal{E}(\omega,\xi)]\right)+R(C_{\xi}(t\eta),\mathcal{E}(\omega,\xi))\right]\\
        =&-\dbar(a_{\mathcal{E}(\omega,\xi)}c_{\xi}\eta)+[\mathcal{E}(\omega,\xi),a_{\mathcal{E}(\omega,\xi)}c_{\xi}\eta]-\iota_{\overline{c_{\xi}\eta}}\left(
        \dbar\mathcal{E}(\omega,\xi)-\tfrac{1}{2}[\mathcal{E}(\omega,\xi),\mathcal{E}(\omega,\xi)]\right),
    \end{align*}
    where $c_{\xi}(\eta)=D C_{\xi}(0)\eta$. After setting $b_{(\omega,\xi)}\eta=a_{\mathcal{E}(\omega,\xi)}c_{\xi}\eta$ we arrive at the expression as stated in Proposition~\ref{prop: DEw}.
\end{proof}

\begin{lemma}
\label{lem: properties b and c}
    For some $\varepsilon>0$, let $\mathcal{W}^{(0)}_{\varepsilon}$, $\mathcal{U}^{(1)}_{\varepsilon}$ as above. We assume that $\varepsilon>0$ is small enough so that $c_{\xi}$ and $b_{(\omega,\xi)}$ as in Proposition~\ref{prop: DEw} are defined for $\xi\in\mathcal{U}^{(1)}_{\varepsilon}$ and $\omega\in\mathcal{W}^{(0)}_{\varepsilon}$. Then the families of linear maps
    \begin{align*}
        c\colon \mathcal{U}^{(1)}_{\varepsilon}\times \A^0(M,T^{1,0}) &\longrightarrow \A^0(M,T^{1,0})\\
        (\xi,\eta)&\longmapsto c_{\xi}\eta 
    \end{align*}
    and 
    \begin{align*}
        b\colon \mathcal{W}^{(0)}_{\varepsilon}\times \mathcal{U}^{(1)}_{\varepsilon}\times\A^0(M,T^{1,0}) &\longrightarrow \A^0(M,T^{1,0})\\
        (\omega,\xi,\eta)&\longmapsto b_{(\omega,\xi)}\eta 
    \end{align*}
    are smooth tame. Moreover, for $\varepsilon>0$ small enough, $c_{\xi}$ is invertible for every $\xi\in \mathcal{U}^{(1)}_{\varepsilon}$, and $b_{(\omega,\xi)}$ is invertible for every $\xi\in \mathcal{U}^{(2)}_{\varepsilon}$, $\omega\in\mathcal{W}^{(1)}_{\varepsilon}$. The families of inverses $c_{\xi}^{-1}$ and $b_{(\omega,\xi)}^{-1}$ are smooth tame. 
\end{lemma}
\begin{proof}
    Recall that
    \begin{align*}
    c_{\xi}\eta=DC_{\xi}(0)\eta,\quad C_{\xi}\colon \mathcal{U}\ni\tilde{\eta}\mapsto \left(\exp^h\circ \Psi\right)^{-1}\left(F_{\xi+\tilde{\eta}} F_{\xi}^{-1}\right)\in \A^0\left(M,T^{1,0}\right),
    \end{align*}
   where $\mathcal{U}\subset \mathcal{U}_{\varepsilon}^{(1)}$ is an open neighborhood of $0$ such that $\xi+\tilde{\eta}\in\mathcal{U}_{\varepsilon}^{(1)}$ for every $\tilde{\eta}\in\mathcal{U}$ and $F_{\xi+\tilde{\eta}}\circ F_{\xi}^{-1}\in \exp^h\circ \Psi(\mathcal{U}_{\varepsilon}^{(1)})$. Note that $C_{\xi}$ can be extended to a $C^0$-neighborhood $\tilde{\mathcal{U}}$ of $0$ such that $\xi+\tilde{\eta}\in(\exp^h\circ \Psi)^{-1}(\mathcal{V}_{\id})$ and $F_{\xi+\tilde{\eta}}F_{\xi}^{-1}\in\mathcal{V}_{\id}$ for every $\tilde{\eta}\in\tilde{\mathcal{U}}$, where $\mathcal{V}_{\id}$ denotes the local chart neighborhood of $\id\in\Diff(M)$. We may furthermore assume, possibly after shrinking $\varepsilon>0$, that $\tilde{\mathcal{U}}$ can be chosen independently of $\xi$, such that we obtain a well-defined map
   \begin{align*}
       C\colon \mathcal{U}_{\varepsilon}^{(1)}\times \tilde{\mathcal{U}}&\longrightarrow\A^0(M,T^{1,0})\\
       (\xi,\tilde{\eta})&\longmapsto C_{\xi}(\tilde{\eta}).
   \end{align*}
   Then $C_{\xi}(\tilde{\eta})$ is smooth tame in $\xi$ and $\tilde{\eta}$ as a composition of smooth tame maps and invertible for every small enough $\xi$. To compute the inverse, we note that
   \begin{align*}
       C_{\xi}(\tilde{\eta})=(\exp^h\circ \Psi)^{-1}\left(R_{F_{\xi}^{-1}}\circ\exp^h\circ \Psi(\xi+\tilde{\eta})\right).
   \end{align*}
   Since $R_{F_{\xi}^{-1}}^{-1}=R_{F_{\xi}}$, we find that
    \begin{align*}
        C_{\xi}^{-1}(\chi)&=(\exp^h\circ\Psi)^{-1}\circ R_{F_{\xi}}\circ \exp^h\circ \Psi(\chi)-\xi
        \\&=(\exp^h\circ \Psi)^{-1}\left(F_{\chi}F_{\xi}\right)-\xi,
    \end{align*}
    which is smooth tame. It follows that the family $c$ of linear maps is smooth tame and $c_{\xi}$ is invertible for small enough $\xi$ and the family of inverses $c_{\xi}^{-1}$ is again smooth tame. 

    For the statements about $b$ we observe that when shrinking $\varepsilon$, for $(\omega,\xi)\in\mathcal{W}^{(1)}_{\varepsilon}\times \mathcal{U}^{(2)}_{\varepsilon}$ we can make $\Vert \mathcal{E}(\omega,\xi)\Vert_{0,\infty}$ arbitrarily small. Indeed, using Proposition~\ref{prop: int equation DP} and the properties of the Fréchet integral (cf.\ Theorem I.2.2.1 in \cite{hamilton82}), we find
    \begin{align*}
        \Vert \mathcal{E}(\omega,\xi)\Vert_{0,\infty}&=\left\Vert \omega +\int_0^1 D_{\xi}\mathcal{E}(\omega,t\xi)\xi\,\dd t\right\Vert_{0,\infty}\\
        &\leq \Vert \omega\Vert_{0,\infty}+\int_0^1\left\Vert D_{\xi}\mathcal{E}(\omega,t\xi)\xi\right\Vert_{0,\infty}\,\dd t\\
        &\leq \Vert \omega\Vert_{0,\infty}+ \int_0^1 C\left((\Vert \omega\Vert_{1,\infty}+t\Vert \xi\Vert_{2,\infty}\right)\Vert \xi\Vert_{1,\infty}+\Vert \xi\Vert_{1,\infty})\,\dd t\\
        &\leq \varepsilon\left(1+C+\tfrac{3}{2}\varepsilon\right),
    \end{align*}
    for $(\omega,\xi)\in\mathcal{W}^{(1)}_{\varepsilon}\times \mathcal{U}^{(2)}_{\varepsilon}$, where we used Lemma~\ref{lem: tame estimate family lin maps} in the third step.
    
    It follows that for small enough $\varepsilon>0$, we can use Lemma~\ref{lem: properties a_omega} to conclude that $a_{\mathcal{E}(\omega,\xi)}\eta$ is a smooth tame family of linear maps which is invertible for every $\omega\in\mathcal{W}_{\varepsilon}^{(1)}$, $\xi\in \mathcal{U}_{\varepsilon}^{(2)}$ and the family of inverses is again smooth tame. Since $b_{(\omega,\xi)}\eta=a_{\mathcal{E}(\omega,\xi)}c_{\xi}\eta $, this proves the assertion.
    \end{proof}

\section{The deformation theorem}
\label{sect: def thm}
Let now $(M,J)$ be a compact complex manifold. When not stated otherwise, the bi-gradings appearing in this section are with respect to the complex structure $J$. Our aim is to show the following:

\begin{theorem}[Deformation theorem]
\label{thm: deformation theorem}
There exists an open neighborhood $\mathcal{W}\subset\mathcal{H}^{0,1}(M,T^{1,0})$ of $0$, a family $\tilde{\mathcal{M}}=\{\omega_t:t\in \mathcal{W}\}$ of almost complex
deformations of $J$, and an analytic obstruction map 
\begin{equation*}
\Phi\colon\mathcal{W}\rightarrow \mathcal{H}^{0,2}(M,T^{1,0})  \end{equation*}
such that the almost complex deformations $\mathcal{M}:=\{\omega_t\in\tilde{\mathcal{M}}:\Phi(t)=0\}$ are precisely the integrable ones. Any sufficiently small complex deformation of $J$ is equivalent to at least one member
of the family $\mathcal{M}$. In the case that the obstruction map vanishes, $\mathcal{M}$ is a locally complete family of complex deformations.    
\end{theorem}

The proof of Theorem~\ref{thm: deformation theorem} is split into two parts. First, we will construct the family and the obstruction map in Section~\ref{sect: descr family}. Then we will prove that the family is locally complete in Section~\ref{sect: loc cplt family}. We will closely follow Kuranishi's line of thought (in particular in Section~\ref{sect: descr family}) complemented by the techniques of Hamilton--Nash--Moser theory.

\subsection{Description of the family}
\label{sect: descr family}

Let 
\begin{align*}
    \tilde{\mathcal{V}}_1:=\left\{\omega\in \A^{0,1}(M,T^{1,0}):\dbar \omega -\tfrac{1}{2}[\omega,\omega]=0=\dbar^*\omega\right\}\subset \A^{0,1}(M,T^{1,0}).
\end{align*}
We will see that as a set, the family of complex deformations $\mathcal{M}$ is contained in $\tilde{\mathcal{V}}_1$. Our strategy will be the following: First we show that $\tilde{\mathcal{V}}_1$ is contained in another set $\tilde{\mathcal{V}}_2$ and that for small enough $\varepsilon>0$ and large enough $k\in \N_0$, the intersection $\mathcal{W}_{\varepsilon}^{(k)}\cap\tilde{\mathcal{V}}_2$ is parametrized by a neighborhood of 0 in $\mathcal{H}^{0,1}(M,T^{1,0})$. Then we will find an analytic map $$\phi\colon\mathcal{W}_{\varepsilon}^{(k)}\cap\tilde{\mathcal{V}}_2\longrightarrow \mathcal{H}^{0,2}(M,T^{1,0})$$ such that $\phi(\varphi)=0$ if and only if $\varphi$ is integrable and hence contained in $\mathcal{M}$.

In the following lemma, we define the set $\tilde{\mathcal{V}}_2$ and prove that it contains the set $\tilde{\mathcal{V}}_1$ as defined above.

\begin{lemma}
    $\tilde{\mathcal{V}}_1\subset \tilde{\mathcal{V}}_2:=\left\{\omega\in \A^{0,1}(M,T^{1,0}):\omega- \tfrac{1}{2}\Q[\omega,\omega]\in\mathcal{H}^{0,1}(M,T^{1,0})\right\}$.
\end{lemma}
\begin{proof}
 Let $\omega\in \tilde{\mathcal{V}}_1$. Then
\begin{align*}
    \Delta \omega =(\dbar\dbar^*+\dbar^*\dbar)\omega=\dbar^*\dbar\omega=\tfrac{1}{2}\dbar^*[\omega,\omega],
\end{align*}
where we used $\dbar^*\omega=0$ for the second equality and $\dbar\omega=\tfrac{1}{2}[\omega,\omega]$ for the last one. Applying the Green's operator $\G$ to both sides and using $\Hproj+\G\Delta=\id$ yields
\begin{align*}
    0=\G\Delta\omega-\tfrac{1}{2}\G\dbar^*[\omega,\omega]=\omega-\Hproj\omega-\tfrac{1}{2}\Q[\omega,\omega].
\end{align*}
Hence,
\begin{align*}
    \omega-\tfrac{1}{2}\Q[\omega,\omega]=\Hproj \omega\in \mathcal{H}^{0,1}(M,T^{1,0})
\end{align*}
and thus, $\omega\in \tilde{\mathcal{V}}_2$.
\end{proof}

Next, we want to show that for some $\varepsilon>0$ and $k\in\N_0$, the intersection $\mathcal{W}_{\varepsilon}^{(k)}\cap\tilde{\mathcal{V}}_2$ is parametrized by a neighborhood of 0 in $\mathcal{H}^{0,1}(M,T^{1,0})$. To this end, we define the map
\begin{align*}
    \sigma\colon \A^{0,1}(M,T^{1,0})&\longrightarrow\A^{0,1}(M,T^{1,0})\\
    \omega&\longmapsto \omega- \tfrac{1}{2}\Q[\omega,\omega].
\end{align*}

\begin{lemma}
    The map $\sigma$ is smooth tame.
\end{lemma}
\begin{proof}
This follows from Example~\ref{ex: bracket} and \ref{ex: Q}.
\end{proof}

\begin{proposition}
\label{prop: existence VF}
 For $\varepsilon>0$ small enough and large enough $k\in\N_0$, the equation $D\sigma(\omega)\alpha=\beta$ has a unique solution $\alpha$ for every $\omega\in \mathcal{W}^{(k)}_{\varepsilon}$, $\beta\in \A^{0,1}(M,T^{1,0})$. Moreover, the family $V(\omega)\beta=\alpha$ of solutions is a smooth tame family of linear maps.
\end{proposition}
\begin{proof}
We compute
\begin{align*}
D\sigma(\omega)\alpha=\alpha-\Q[\alpha,\omega]=\alpha-\dbar^*\G[\alpha,\omega].
\end{align*}
The statement follows then from Proposition~\ref{prop: inverse family from elliptic}.
\end{proof}

\begin{corollary}
    For $\varepsilon>0$  small enough and $k\in\N_0$ large enough, $\sigma$ maps $\mathcal{W}_{\varepsilon}^{(k)}$ bijectively onto a neighborhood $\tilde{\mathcal{W}}_{\varepsilon}\subset \A^{0,1}(M,T^{1,0})$ of $0$ such that the inverse map
    \begin{align*}
        \sigma^{-1}\colon \tilde{\mathcal{W}}_{\varepsilon}\longrightarrow \mathcal{W}^{(k)}_{\varepsilon}
    \end{align*}
    is smooth tame.
\end{corollary}
\begin{proof}
    This follows from Theorem~\ref{thm: nash moser} and Proposition~\ref{prop: existence VF}.
\end{proof}

In particular, setting $\mathcal{W}_{\varepsilon}:=\tilde{\mathcal{W}}_{\varepsilon}\cap \mathcal{H}^{0,1}(M,T^{1,0})$ and $\mathcal{V}_{\varepsilon}:=\tilde{\mathcal{V}}_2\cap \mathcal{W}^{(k)}_{\varepsilon}$, we obtain a smooth tame map
\begin{align*}
    \varphi:=\sigma^{-1}|_{\mathcal{W}_{\varepsilon}}\colon \mathcal{W}_{\varepsilon}\longrightarrow \mathcal{V}_{\varepsilon}.
\end{align*}

\begin{corollary}
\label{cor: hol family}
The family $\{ \varphi(t): t\in \mathcal{W}_{\varepsilon}\}$ is a smooth family of almost complex deformations.    
\end{corollary}

Next, we will identify the obstruction map, whose vanishing set consists of the integrable deformations in $\mathcal{V}_{\varepsilon}$.

\begin{lemma}
\label{lem: analytic obstruction}
    For $\varepsilon>0$ small enough, $\varphi\in \mathcal{V}_{\varepsilon}$ is in $\mathcal{S}_{\varepsilon}:=\tilde{\mathcal{V}}_1\cap \mathcal{W}^{(k)}_{\varepsilon}$ if and only if $\Phi(\varphi):=\Hproj[\varphi,\varphi]=0$.
\end{lemma}
\begin{proof}
    We first note that for any $\varphi\in \mathcal{V}_{\varepsilon}$, we have 
    \begin{align*}
        \dbar\varphi-\tfrac{1}{2}[\varphi,\varphi]=&\tfrac{1}{2}\dbar \Q[\varphi,\varphi]-\tfrac{1}{2}[\varphi,\varphi]\\
        =&-\tfrac{1}{2}\Q \dbar [\varphi,\varphi]-\tfrac{1}{2}\Hproj[\varphi,\varphi],
    \end{align*}
    where we used $\varphi\in \mathcal{V}_{\varepsilon}$ for the first equality and $\Hproj+\dbar\Q+\Q\dbar=\id$ for the second equality. Since the images of $\Q$ and $\Hproj$ are orthogonal to each other, it follows that 
    \begin{align}
    \label{eq: integrability}    \dbar\varphi-\tfrac{1}{2}[\varphi,\varphi]=0\quad \iff \quad \Q\dbar[\varphi,\varphi]=0\text{ and }\Hproj[\varphi,\varphi]=0.
    \end{align}
    We claim that for $\varepsilon>0$ small enough, $\Hproj [\varphi,\varphi]=0$ already implies $\Q\dbar[\varphi,\varphi]=0$. To see this, suppose $\varphi\in\mathcal{V}_{\varepsilon}$ such that $\Hproj[\varphi,\varphi]=0$. Then,
    \begin{align*}
        \Q\dbar[\varphi,\varphi]=-\dbar\Q[\varphi,\varphi]+[\varphi,\varphi].
    \end{align*}
    Using the compatibility of the bracket $[\cdot,\cdot]$ and $\dbar$ as stated in Lemma~\ref{lem: properties dbar and bracket} and $\varphi\in \tilde{\mathcal{V}}_2$, we compute
    \begin{align*}
        \Q\dbar[\varphi,\varphi]=2\Q[\dbar\varphi,\varphi]=\Q[\dbar\Q[\varphi,\varphi],\varphi]=\Q[-\Q\dbar[\varphi,\varphi],\varphi]+\Q[[\varphi,\varphi],\varphi].
    \end{align*}
    By the Jacobi identity (see Lemma~\ref{lem: properties dbar and bracket}), we must have $\left[[\varphi,\varphi],\varphi\right]=0$. Therefore, setting $\eta:=\Q\dbar[\varphi,\varphi]$, we find
    \begin{align*}
        \eta=-\Q[\eta,\varphi].
    \end{align*}
    It follows from Example~\ref{ex: bracket} and~\ref{ex: Q} that (for $k\geq \dim M/2+1$)
    \begin{align*}
        \Vert \eta\Vert_{k,2}=\Vert \Q[\eta,\varphi]\Vert_{0,2}\leq C\Vert\eta\Vert_{k,2}\Vert\varphi\Vert_{k,2}\leq C\varepsilon\Vert\eta\Vert_{k,2}.
    \end{align*}
    In particular, for $\varepsilon <1/C$, this implies that $\eta=0$ and hence $\Q\dbar[\varphi,\varphi]=0$. By equation~\eqref{eq: integrability} this shows that $\varphi\in\mathcal{S}_{\varepsilon}$.
\end{proof}

\subsection{Local completeness of the family}
\label{sect: loc cplt family}
In the last step of the proof of Theorem \ref{thm: deformation theorem}, we need to show that every sufficiently small complex deformation $\omega$ can be obtained from $\mathcal{M}$ via a diffeomorphism, i.e.\ there is some small $\xi\in \A^0(M,T^{1,0})$ such that $\dbar^*\mathcal{E}(\omega,\xi)=0$. Moreover, given a smooth family of small (almost) complex deformations $\omega_t$, we need to show that there is a family $\xi_t\subset \A^0(M,T^{1,0})$ depending smoothly on $t$ such that $\dbar^*\mathcal{E}(\omega_t,\xi_t)=0$.

In the following, we will denote by ${\mathcal{H}^0(M,T^{1,0})}^{\perp}$, the orthogonal complement of $\mathcal{H}^0(M,T^{1,0})$ in $\A^0(M,T^{1,0})$. The elements of ${\mathcal{H}^0(M,T^{1,0})}^{\perp}$ are the $\dbar^*$-exact forms in $\A^0(M,T^{1,0})$.
\begin{proposition}
\label{prop: locally complete}
    There is a neighborhood $\mathcal{W}$ of the origin in $\A^{0,1}(M,T^{1,0})$ and a neighborhood $\mathcal{U}$ of the origin in ${\mathcal{H}^0(M,T^{1,0})}^{\perp}$ such that for any $\omega\in \mathcal{W}$ there is a unique solution $\xi=\xi(\omega)\in \mathcal{U}$ to the equation
    \begin{align*}
        \dbar^*\mathcal{E}(\omega,\xi)=0.
    \end{align*} The solution map $\xi(\omega)$ is smooth tame.
\end{proposition}

\begin{rem}
    Note that if $\mathcal{H}^0(M,T^{1,0})=0$, Proposition~\ref{prop: locally complete} implies that every almost complex deformation $\omega\in\mathcal{W}$ is related to an almost complex deformation $\omega'$ such that $\dbar^*\omega'=0$ via a unique diffeomorphism $F_{\xi}$, $\xi\in\mathcal{U}$. Thus, the representative $\omega'$ is locally unique. If the obstruction map vanishes, Theorem~\ref{thm: deformation theorem} yields a so-called \emph{miniversal} family of complex deformations and the condition $\dbar^*\omega=0$ can be thought of as taking a local slice for the action of the diffeomorphisms on the space of (almost) complex structures.

    Since $H^0(M,T^{1,0})\cong \mathcal{H}^0(M,T^{1,0})$ is the Lie algebra of the Lie group of complex automorphisms of $M$, the condition $H^0(M,T^{1,0})=0$ is satisfied if and only if the group of complex automorphisms of $M$ is discrete.
\end{rem}

For the proof of Proposition~\ref{prop: locally complete}, we wish to apply Theorem~\ref{thm: impl
fn thm}. For this, note that $\dbar^*\mathcal{E}(\omega,\xi)=0$ if and only if $\Q\mathcal{E}(\omega,\xi)=0$. For some $\varepsilon>0$, $k\in\N$, let ${\mathcal{W}}^{(k)}_{\varepsilon}$ be as before and 
\begin{align*}
    \tilde{\mathcal{U}}^{(k)}_{\varepsilon}:=\left\{\xi\in {\mathcal{H}^0(M,T^{1,0})}^{\perp}:\Vert\xi \Vert_k<\varepsilon\right\}.
\end{align*}
Then, for $\varepsilon$ small enough and large enough $k$ the map
\begin{align*}
    A\colon \tilde{\mathcal{U}}^{(k)}_{\varepsilon}\times{\mathcal{W}}^{(k)}_{\varepsilon}&\longrightarrow {\mathcal{H}^0(M,T^{1,0})}^{\perp}\\
    (\xi,\omega)&\longmapsto \Q\mathcal{E}(\omega,\xi)
\end{align*}
is well-defined and smooth tame as a composition of smooth tame maps. Using Proposition~\ref{prop: DEw}, for $(\xi,\omega)\in \tilde{\mathcal{U}}_{\varepsilon}\times {\mathcal{W}}_{\varepsilon}$, $\alpha\in{\mathcal{H} ^0(M,T^{1,0})}^{\perp}$ we compute 
\begin{align*}
    D_{\xi}A(\xi,\omega)\alpha=&\Q D_{\xi}\mathcal{E}(\omega,\xi)\alpha
    \\=&-\Q\dbar(b_{(\omega,\xi)}\alpha)+\Q[\mathcal{E}(\omega,\xi),b_{(\omega,\xi)}\alpha]-\Q \iota_{\overline{c_{\xi}\alpha}}\left(\dbar\mathcal{E}(\omega,\xi)-\tfrac{1}{2}[\mathcal{E}(\omega,\xi),\mathcal{E}(\omega,\xi)]\right)\\
    =&-(\id-\Hproj)(b_{(\omega,\xi)}\alpha)+\Q[\mathcal{E}(\omega,\xi),b_{(\omega,\xi)}\alpha]-\Q \iota_{\overline{c_{\xi}\alpha}}\left(\dbar\mathcal{E}(\omega,\xi)-\tfrac{1}{2}[\mathcal{E}(\omega,\xi),\mathcal{E}(\omega,\xi)]\right),
\end{align*}    
where we used Proposition~\ref{prop: H G and Q} in the last step as well as the fact that $\Q$ maps $\A^0(M,T^{1,0})$ to $0$.

\begin{lemma}
    For $\varepsilon>0$ small enough, the map
    \begin{align*}
   (\id-\Hproj)\circ b_{(\omega,\xi)}\colon {\mathcal{H} ^0(M,T^{1,0})}^{\perp}\longrightarrow {\mathcal{H} ^0(M,T^{1,0})}^{\perp}
\end{align*}
is invertible for every $(\xi,\omega)\in\tilde{\mathcal{U}}_{\varepsilon}\times \mathcal{W}_{\varepsilon}$. The family of inverses is a smooth tame family of linear maps.
\end{lemma}
\begin{proof}
    We follow similar ideas as in the proof of \cite[Lem III.2.6]{Omori2016}. Recall that for $\eta\in \mathcal{A}^0(M,T^{1,0})$, the map $b_{(\omega,\xi)}$ was defined as
    \begin{align*}
        b_{(\omega,\xi)}\eta=c_{\xi}\eta+\iota_{\overline{c_{\xi}\eta}}\mathcal{E}(\omega,\xi),
    \end{align*}
    where
    \begin{align*}
    c_{\xi}\eta=DC_{\xi}(0)\eta,\quad C_{\xi}\colon \mathcal{U}\ni\tilde{\eta}\mapsto \left(\exp^h\circ \Psi\right)^{-1}\left(F_{\xi+\tilde{\eta}} F_{\xi}^{-1}\right)\in \A^0\left(M,T^{1,0}\right),
    \end{align*}
    for a sufficiently small $C^0$-neighborhood $\mathcal{U}\subset \A^0\left(M,T^{1,0}\right)$ of zero. Recall furthermore that by Lemma~\ref{lem: properties b and c}, for $\varepsilon>0$ small enough $b_{(\omega,\xi)}$ is invertible for every $(\xi,\omega)\in\tilde{\mathcal{U}}_{\varepsilon}\times \mathcal{W}_{\varepsilon}$ and the family of inverses $b_{(\omega,\xi)}^{-1}$ is smooth tame. This together with the observation that
    \begin{align*}
     \mathcal{H}^0(M,T^{1,0})\cap b_{(\omega,\xi)}\left({\mathcal{H} ^0(M,T^{1,0})}^{\perp}\right)=\{0\}   
    \end{align*}
    for all $(\xi,\omega)\in\tilde{\mathcal{U}}_{\varepsilon}\times \mathcal{W}_{\varepsilon}$ and $\varepsilon>0$ small enough shows that 
    \begin{align*}
   (\id-\Hproj)\circ b_{(\omega,\xi)}\colon {\mathcal{H} ^0(M,T^{1,0})}^{\perp}\longrightarrow {\mathcal{H} ^0(M,T^{1,0})}^{\perp}
\end{align*}
is injective. To show that it is surjective, we introduce several auxiliary families of linear maps. First, we observe that, again for $\varepsilon>0$ small enough and $(\xi,\omega)\in\tilde{\mathcal{U}}_{\varepsilon}\times \mathcal{W}_{\varepsilon}$,
\begin{align*}
    b_{(\omega,\xi)}^{-1}\left(\mathcal{H}^0(M,T^{1,0})\right)\cap \mathcal{H}^0(M,T^{1,0})^{\perp}=\{0\}.
\end{align*}
Since $\mathcal{H}^0(M,T^{1,0})$ is finite-dimensional, this implies that 
\begin{align*}
    \Hproj b_{(\omega,\xi)}^{-1}\colon \mathcal{H}^0(M,T^{1,0})\longrightarrow\mathcal{H}^0(M,T^{1,0})
\end{align*}
is invertible and we denote the inverse map by $q_{(\omega,\xi)}$. Being a smooth map to a finite dimensional vector space, $q_{(\omega,\xi)}$ is a smooth tame family of linear maps. We furthermore define 
\begin{align*}
    \psi_1(\omega,\xi)&\colon \mathcal{A}^0(M,T^{1,0})\longrightarrow \mathcal{H}^0(M,T^{1,0})^{\perp},\quad \eta\longmapsto \eta-b_{(\omega,\xi)}^{-1}q_{(\omega,\xi)}\Hproj \eta,\\
    \psi_2(\omega,\xi)&\colon \mathcal{A}^0(M,T^{1,0})\longrightarrow b_{(\omega,\xi)}^{-1}\left(\mathcal{H}^0(M,T^{1,0})\right),\quad \eta\longmapsto b_{(\omega,\xi)}^{-1}q_{(\omega,\xi)}\Hproj \eta,\\
    \varphi(\omega,\xi)&\colon \mathcal{A}^0(M,T^{1,0})\longrightarrow \mathcal{H}^0(M,T^{1,0}),\quad \eta \longmapsto \Hproj b_{(\omega,\xi)}\psi_1(\omega,\xi)b_{(\omega,\xi)}^{-1}\eta. 
\end{align*}
We claim that 
\begin{align*}
    b_{(\omega,\xi)}\left(\mathcal{H}^0(M,T^{1,0})^{\perp}\right)=\left\{\eta +\varphi(\omega,\xi)\eta:\eta\in\mathcal{H}^0(M,T^{1,0})^{\perp}\right\}.
\end{align*}
To see this, given $\eta\in \mathcal{H}^0(M,T^{1,0})^{\perp}$, we write
\begin{align*}
    b_{(\omega,\xi)}^{-1}\eta=\psi_1(\omega,\xi)b_{(\omega,\xi)}^{-1}\eta+\psi_2(\omega,\xi)b_{(\omega,\xi)}^{-1}\eta.  
\end{align*}
This implies
\begin{align*}
    \eta&=(\id-\Hproj)b_{(\omega,\xi)}\psi_1(\omega,\xi)b_{(\omega,\xi)}^{-1}\eta+\Hproj b_{(\omega,\xi)}\psi_1(\omega,\xi)b_{(\omega,\xi)}^{-1}\eta\\
    &\quad\quad +b_{(\omega,\xi)}\psi_2(\omega,\xi)b_{(\omega,\xi)}^{-1}\eta.
\end{align*}
Since the last two terms on the right-hand side are in $\mathcal{H}^0(M,T^{1,0})$ and $\eta\in \mathcal{H}^0(M,T^{1,0})^{\perp}$ it follows that
\begin{align*}
    \eta+\Hproj b_{(\omega,\xi)}\psi_1(\omega,\xi)b_{(\omega,\xi)}^{-1}\eta=b_{(\omega,\xi)}\psi_1(\omega,\xi)b_{(\omega,\xi)}^{-1}\eta
\end{align*}
and hence $\eta+\varphi(\omega,\xi)\eta\in b_{(\omega,\xi)}\left(\mathcal{H}^0(M,T^{1,0})^{\perp}\right)$. Therefore, we have
\begin{align*}
    \left\{\eta+\varphi(\omega,\xi)\eta:\eta\in \mathcal{H}^0(M,T^{1,0})^{\perp}\right\}\subset b_{(\omega,\xi)}\left(\mathcal{H}^0(M,T^{1,0})^{\perp}\right).
\end{align*}
Observing that the subspaces on both sides have the same (finite) codimension, we conclude that they are in fact equal.

In particular, this shows that $(\id-\Hproj)b_{(\omega,\xi)}\colon \mathcal{H}^0(M,T^{1,0})^{\perp}\rightarrow \mathcal{H}^0(M,T^{1,0})^{\perp}$ is surjective and therefore invertible. The family of inverses is given by
\begin{align*}
    p_{(\omega,\xi)}\colon \mathcal{H}^0(M,T^{1,0})^{\perp}&\longrightarrow \mathcal{H}^0(M,T^{1,0})^{\perp}\\
    \eta &\longmapsto b_{(\omega,\xi)}^{-1}(\eta+\varphi(\omega,\xi)\eta)
\end{align*}
which is smooth tame as a composition of smooth tame maps.
\end{proof}

\begin{lemma}
\label{lem: DA invertible, locally complete}
    For $\varepsilon>0$ small enough and large enough $k$, the equation $D_{\xi}A(\xi,\omega)\alpha=\beta$ has a unique solution $\alpha$ for every $(\xi,\omega)\in \tilde{\mathcal{U}}_{\varepsilon}^{(k)}\times {\mathcal{W}}_{\varepsilon}^{(k)}$, $\beta\in {\mathcal{H}^0(M,T^{1,0})}^{\perp}$. Moreover, the family of inverse maps is a smooth tame family of linear maps
\end{lemma}

\begin{proof}

Setting $\tilde{\alpha}:=(\id-\Hproj)b_{(\omega,\xi)}\alpha\in\mathcal{H}^0(M,T^{1,0})^{\perp}$, we observe that finding a solution to the equation $D_{\xi}A(\xi,\omega)\alpha=\beta$ is equivalent to finding a solution of to the equation
\begin{align*}
  \tilde{\alpha}-\Q[\mathcal{E}(\omega,\xi),b_{(\omega,\xi)}p_{(\omega,\xi)}\tilde{\alpha}]+\Q \iota_{\overline{c_{\xi}p_{(\omega,\xi)}\tilde{\alpha}}}\left(\dbar\mathcal{E}(\omega,\xi)-\tfrac{1}{2}[\mathcal{E}(\omega,\xi),\mathcal{E}(\omega,\xi)]\right)=\beta.
\end{align*}
The statement follows then from Proposition~\ref{prop: inverse family from elliptic}
\end{proof}

\begin{proof}[Proof of Proposition~\ref{prop: locally complete}]
By Lemma~\ref{lem: DA invertible, locally complete}, the assumptions of Theorem~\ref{thm: impl fn thm} hold. Since $A(0,0)=0$, we conclude that there is a neighborhood $\mathcal{W}$ of $0$ in $\A^{0,1}(M,T^{1,0})$ and a neighborhood $\mathcal{U}$ of $0$ in ${\mathcal{H}^0(M,T^{1,0})}^{\perp}$ such that for every $\omega\in\mathcal{W}$ we can find some $\xi=\xi(\omega)\in\mathcal{U}$ such that $\dbar^*\mathcal{E}(\omega,\xi)=0$. Since the derivative $D_{\xi}A(\xi,\omega)$ is even invertible, it follows furthermore from Theorem~\ref{thm: nash moser} that the solution $\xi(\omega)$ is (locally) unique (see also \cite[Thm A.12]{Donaldson2025}). Moreover, the solution $\xi(\omega)$ is defined by a smooth tame map. This finishes the proof of Proposition~\ref{prop: locally complete}.
\end{proof}

\begin{proof}[Proof of Theorem~\ref{thm: deformation theorem}]
By Corollary~\ref{cor: hol family}, $\{ \varphi(t): t\in \mathcal{W}_{\varepsilon}\}$ is a smooth family of almost complex deformations. The integrable deformations are precisely those for which the analytic map $\Phi(\alpha):=\Hproj[\alpha,\alpha]$ vanishes. They are contained in a neighborhood of 0 in $\tilde{\mathcal{V}}_1$. By Proposition~\ref{prop: locally complete}, we know that any sufficiently small complex deformation is equivalent to a complex deformation in a neighborhood of 0 of $\tilde{\mathcal{V}}_1$. 

Furthermore, if $\Phi\equiv 0$, $\{ \varphi(t): t\in \mathcal{W}_{\varepsilon}\}$ is a family of complex deformations. Given any other family of complex deformations $\{\varphi'(t'):t'\in \mathcal{W}'
\}$, it follows again from Proposition~\ref{prop: locally complete} that (possibly after shrinking $\mathcal{W}'$) there is a family of diffeomorphisms $\{f_{t'}\}_{t'\in \mathcal{W}'}$, depending smoothly on $t'$, and a smooth map $\tau\colon \mathcal{W}'\rightarrow \mathcal{W}_{\varepsilon}$ such that for every $t'\in \mathcal{W}'$, we have
\begin{align*}
    f_{t'}\cdot\varphi'(t') =\varphi(\tau(t')).
\end{align*}
Hence, the family is locally complete.
\end{proof}

\singlespacing\small
\bibliographystyle{plain}
\bibliography{references}

@article{hamilton82,
  title = {The inverse function theorem of Nash and Moser},
  volume = {7},
  ISSN = {0273-0979},
  url = {http://dx.doi.org/10.1090/S0273-0979-1982-15004-2},
  DOI = {10.1090/s0273-0979-1982-15004-2},
  number = {1},
  journal = {Bulletin of the American Mathematical Society},
  publisher = {American Mathematical Society (AMS)},
  author = {Hamilton,  R. S.},
  year = {1982},
  pages = {65–222}
}

@article{Kuranishi1962,
  title = {On the Locally Complete Families of Complex Analytic Structures},
  volume = {75},
  ISSN = {0003-486X},
  url = {http://dx.doi.org/10.2307/1970211},
  DOI = {10.2307/1970211},
  number = {3},
  journal = {The Annals of Mathematics},
  publisher = {JSTOR},
  author = {Kuranishi,  M.},
  year = {1962},
  month = May,
  pages = {536}
}

@incollection {kuranishi64,
    AUTHOR = {Kuranishi, M.},
     TITLE = {New proof for the existence of locally complete families of
              complex structures},
 BOOKTITLE = {Proceedings of the Conference on Complex Analysis. ({M}inneapolis, 1964)},
     PAGES = {142--154},
 PUBLISHER = {Springer, Berlin-Heidelberg-New York},
      YEAR = {1965},
   MRCLASS = {32.47 (57.70)},
  MRNUMBER = {176496},
MRREVIEWER = {Phillip\ Griffiths},
}

@article{Frohlicher1957,
  title = {A THEOREM ON STABILITY OF COMPLEX STRUCTURES},
  volume = {43},
  ISSN = {1091-6490},
  url = {http://dx.doi.org/10.1073/pnas.43.2.239},
  DOI = {10.1073/pnas.43.2.239},
  number = {2},
  journal = {Proceedings of the National Academy of Sciences},
  publisher = {Proceedings of the National Academy of Sciences},
  author = {Fr\"{o}licher,  A. and Nijenhuis,  A.},
  year = {1957},
  month = Feb,
  pages = {239–241}
}

@article{kodaira1958completeness,
  title = {A theorem of completeness for complex analytic fibre spaces},
  volume = {100},
  ISSN = {0001-5962},
  url = {http://dx.doi.org/10.1007/BF02559541},
  DOI = {10.1007/bf02559541},
  number = {3-4},
  journal = {Acta Mathematica},
  publisher = {International Press of Boston},
  author = {Kodaira,  K. and Spencer,  D. C.},
  year = {1958},
  pages = {281–294}
}

@article{kodaira1958existence,
  title = {On the Existence of Deformations of Complex Analytic Structures},
  volume = {68},
  ISSN = {0003-486X},
  url = {http://dx.doi.org/10.2307/1970256},
  DOI = {10.2307/1970256},
  number = {2},
  journal = {The Annals of Mathematics},
  publisher = {JSTOR},
  author = {Kodaira,  K. and Nirenberg,  L. and Spencer,  D. C.},
  year = {1958},
  month = {Sept},
  pages = {450}
}

@article {KSdefIandII,
    AUTHOR = {Kodaira, K. and Spencer, D. C.},
     TITLE = {On deformations of complex analytic structures. {I}, {II}},
   JOURNAL = {Annals of Mathematics. Second Series},
    VOLUME = {67},
      YEAR = {1958},
     PAGES = {328--466},
      ISSN = {0003-486X},
   MRCLASS = {57.00},
  MRNUMBER = {112154},
MRREVIEWER = {M.\ F.\ Atiyah},
       DOI = {10.2307/1970009},
       URL = {https://doi.org/10.2307/1970009},
}

@article {KSdefIII,
    AUTHOR = {Kodaira, K. and Spencer, D. C. },
     TITLE = {On deformations of complex analytic structures. {III}.
              {S}tability theorems for complex structures},
  JOURNAL = {Annals of Mathematics. Second Series},
    VOLUME = {71},
      YEAR = {1960},
     PAGES = {43--76},
      ISSN = {0003-486X},
   MRCLASS = {57.00},
  MRNUMBER = {115189},
MRREVIEWER = {M.\ F.\ Atiyah},
       DOI = {10.2307/1969879},
       URL = {https://doi.org/10.2307/1969879},
}

@article{hamilton1977deformation,
  title={Deformation of complex structures on manifolds with boundary. {I}. The stable case},
  author={Hamilton, R. S. },
  journal={Journal of Differential geometry},
  volume={12},
  number={1},
  pages={1--45},
  year={1977},
  publisher={Lehigh University}
}

@BOOK{Wells2007-er,
  title     = "Differential analysis on complex manifolds",
  author    = "Wells, R. O. ",
  publisher = "Springer",
  series    = "Graduate Texts in Mathematics",
  edition   =  3,
  month     =  oct,
  year      =  2007,
  address   = "New York, NY",
  copyright = "https://www.springernature.com/gp/researchers/text-and-data-mining",
  language  = "en"
}

@phdthesis{Gualtieri:2003dx,
    author = "Gualtieri, M.",
    title = "{Generalized complex geometry}",
    eprint = "math/0401221",
    archivePrefix = "arXiv",
    school = "Oxford U.",
    year = "2003"
}

@article{Gualtieri2011,
  title = {Generalized complex geometry},
  volume = {174},
  ISSN = {0003-486X},
  url = {http://dx.doi.org/10.4007/annals.2011.174.1.3},
  DOI = {10.4007/annals.2011.174.1.3},
  number = {1},
  journal = {Annals of Mathematics},
  publisher = {Annals of Mathematics},
  author = {Gualtieri,  M.},
  year = {2011},
  month = {July},
  pages = {75–123}
}

@article{Donaldson2025,
  title = {{C}alabi-{Y}au threefolds with boundary},
  volume = {21},
  ISSN = {1558-8602},
  url = {http://dx.doi.org/10.4310/pamq.250116012309},
  DOI = {10.4310/pamq.250116012309},
  number = {3},
  journal = {Pure and Applied Mathematics Quarterly},
  publisher = {International Press of Boston},
  author = {Donaldson,  S. and Lehmann,  F.},
  year = {2025},
  pages = {1119–1170}
}

@book{Omori2016,
  title = {Infinite-Dimensional Lie Groups},
  ISBN = {9781470445737},
  ISSN = {2472-5137},
  url = {http://dx.doi.org/10.1090/mmono/158},
  DOI = {10.1090/mmono/158},
  journal = {Translations of Mathematical Monographs},
  publisher = {American Mathematical Society},
  author = {Omori,  H.},
  year = {2016},
  month = {Apr} 
}

@book {arnold,
    AUTHOR = {Arnold, V. I. and Gusein-Zade, S. M. and Varchenko, A. N.},
     TITLE = {Singularities of differentiable maps. {V}olume 1},
    SERIES = {Modern Birkh\"auser Classics},
 PUBLISHER = {Birkh\"auser/Springer, New York},
      YEAR = {2012},
     PAGES = {xii+382},
      ISBN = {978-0-8176-8339-9},
   MRCLASS = {58Kxx},
  MRNUMBER = {2896292},
}

\end{document}